\documentclass[12pt]{amsart}
\usepackage{graphicx}
\usepackage{amsmath, color}
\usepackage{amscd, amstext}
\usepackage{amsfonts}
\usepackage{amssymb}
\usepackage{amsthm}

\numberwithin{equation}{section}

\let\secsymb=\S

\newtheorem{theorem}{Theorem}[section]
\newtheorem{lemma}[theorem]{Lemma}
\newtheorem{proposition}[theorem]{Proposition}

\newtheorem{corollary}[theorem]{Corollary}

\theoremstyle{plain}

\newtheorem{prop}[theorem]{Proposition}
\newtheorem{lem}[theorem]{Lemma}
\newtheorem*{lem*}{Lemma}

\theoremstyle{definition}
\newtheorem{rem}[theorem]{Remark}

\newtheorem{eg}[theorem]{Example}

\newtheorem{remarks}[theorem]{Remarks}

\theoremstyle{definition}
\newtheorem{example}[theorem]{Example}
\newtheorem{remark}[theorem]{Remark}

\newcommand{\be}{\begin{equation}}
\newcommand{\ee}{\end{equation}}
\newcommand{\bes}{\begin{equation*}}
\newcommand{\ees}{\end{equation*}}

\newcommand{\cC}{\mathcal{C}}

\newcommand{\cH}{\mathcal{H}}

\newcommand{\cL}{\mathcal{L}}
\newcommand{\cM}{\mathcal{M}}

\newcommand{\cF}{\mathcal{F}}

\newcommand{\cA}{\mathcal{A}}

\newcommand{\cR}{\mathcal{R}}

\newcommand{\lel}{\left\langle}
\newcommand{\rir}{\right\rangle}
\newcommand{\mb}[1]{\mathbb{#1}}

\newcommand{\Aut}{\operatorname{Aut}}

\newcommand{\dist}{\operatorname{dist}}

\newcommand{\Lat}{\operatorname{Lat}}
\newcommand{\Mult}{\operatorname{Mult}}

\newcommand{\spn}{\operatorname{span}}
\newcommand{\id}{\operatorname{id}}
\newcommand{\lip}{\langle}
\newcommand{\rip}{\rangle}
\newcommand{\ip}[1]{\lip #1 \rip}
\newcommand{\bip}[1]{\big\lip #1 \big\rip}

\renewcommand{\phi}{\phi}
\newcommand{\wot}{\textsc{wot}}
\newcommand{\ol}{\overline}

\newcommand{\re}{\operatorname{Re}}

\newcommand{\bB}{{\mathbb{B}}}
\newcommand{\bC}{{\mathbb{C}}}
\newcommand{\bD}{{\mathbb{D}}}

\newcommand{\bT}{{\mathbb{T}}}

\newcommand{\ba}{{\mathbf{a}}}

  \newcommand{\A}{{\mathcal{A}}}
  \newcommand{\B}{{\mathcal{B}}}

  \newcommand{\F}{{\mathcal{F}}}
  
\renewcommand{\H}{{\mathcal{H}}}

  \newcommand{\M}{{\mathcal{M}}}

\renewcommand{\S}{{\mathcal{S}}}

\newcommand{\rC}{{\mathrm{C}}}

\newcommand{\ep}{\varepsilon}
\renewcommand{\phi}{\varphi}
\newcommand{\upchi}{{\raise.35ex\hbox{$\chi$}}}

\newcommand{\qand}{\quad\text{and}\quad}

\newcommand{\qfor}{\quad\text{for}\ }
\newcommand{\qforal}{\quad\text{for all}\ }

\newcommand{\AND}{\text{ and }}

\newcommand{\FORAL}{\text{ for all }}
\newcommand{\IF}{\text{ if }}

\newcommand{\AD}{\mathrm{A}(\mathbb{D})}

\newcommand{\Hinf}{H^\infty}

\begin{document}

\title{Operator algebras for analytic varieties}

\author[K.R. Davidson]{Kenneth R. Davidson}
\address{Department of Pure Mathematics, University of Waterloo.}
\email{krdavids@math.uwaterloo.ca}

\author[C. Ramsey]{Christopher Ramsey}
\email{ciramsey@math.uwaterloo.ca}

\author[O.M. Shalit]{Orr Moshe Shalit}
\address{Department of Mathematics, Ben-Gurion University of the Negev.}
\email{oshalit@math.bgu.ac.il}
\thanks{First and second authors partially supported by NSERC, Canada. \\ \indent
Third author supported by ISF Grant no.
474/12 and by EU FP7/2007-2013 Grant no. 321749}

\dedicatory{Dedicated to the memory of William B. Arveson}

\begin{abstract}
We study the isomorphism problem for the multiplier algebras of irreducible 
complete Pick kernels. 
These are precisely the restrictions $\cM_V$ of the multiplier algebra 
$\cM$ of Drury-Arveson space to a holomorphic subvariety $V$ of the unit ball $\mb{B}_d$.

We find that $\cM_V$ is completely isometrically isomorphic to $\cM_W$ 
if and only if $W$ is the image of $V$ under a biholomorphic automorphism of the ball. 
In this case, the isomorphism is unitarily implemented.
This is then strengthend to show that, when $d<\infty$, every isometric isomorphism is completely isometric.

The problem of characterizing when two such algebras are 
(algebraically) isomorphic is also studied. 
When $V$ and $W$ are each a finite union of irreducible varieties and
a discrete variety in $\mb{B}_d$ with $d<\infty$, then an isomorphism between $\cM_V$ and $\cM_W$ 
determines a biholomorphism (with multiplier coordinates) between the varieties;
and the isomorphism is composition with this function. 
These maps are automatically weak-$*$ continuous.

We present a number of examples showing that the converse fails in several ways. 
We discuss several special cases in which the converse does hold---particularly,
smooth curves and Blaschke sequences.

We also discuss the norm closed algebras associated to a variety,
and point out some of the differences.
\end{abstract}

\subjclass[2010]{47L30, 47A13, 46E22}
\keywords{Non-selfadjoint operator algebras, reproducing kernel Hilbert spaces}
\maketitle

\section{Introduction}

In this paper, we study operator algebras of multipliers on reproducing kernel Hilbert
spaces associated to analytic varieties in the unit ball of $\mb{C}^d$.
The model is the multiplier algebra $\cM_d$ of the Drury-Arveson space,
a.k.a.\ symmetric Fock space. 
The generators, multiplication by coordinate functions, form a $d$-tuple which
is universal for commuting row contractions \cite{Arv98}.
The Hilbert space is a reproducing kernel Hilbert space which is a complete
Nevanlinna-Pick kernel \cite{DavPittsPick}; and in fact when $d=\infty$ is the
universal complete NP kernel \cite{AM00}.
For these reasons, this space and its multiplier algebra have received a lot
of attention in recent years.

In this paper, we are concerned with multipliers on subspaces of Drury-Arveson space
spanned by the kernel functions they contain. By results in \cite{DavPitts2}, these
operator algebras are also complete quotients of $\cM_d$ by \wot-closed ideals.
The zero set is always an analytic variety $V$ in the ball, and the multiplier algebra
$\cM_V$ is a space of holomorphic functions on $V$. 

The main question that we address is when two such algebras are isomorphic.
We find first that two such algebras $\cM_V$ and $\cM_W$ 
are completely isometrically isomorphic if and only if there is a biholomorphic
automorphism of the ball that carries $V$ onto $W$.
In this case, the isomorphism is unitarily implemented.

The question of algebraic isomorphism (which implies continuous algebraic isomorphism
because the algebras are semisimple) is much more subtle, and there are results only for the case $d<\infty$. 
In an earlier paper \cite{DRS}, the authors considered the case of homogeneous varieties. 
We showed, under some extra assumptions on the varieties, that the algebras are 
isomorphic if and only if there is a biholomorphic map of one variety onto the other. 
In a recent paper, Michael Hartz \cite{Hartz} was able to establish this result 
in complete generality.

In this paper, we establish a special case of what should be the easy direction, showing that 
an isomorphism determines a biholomorphism of $V$ onto $W$. This turns
out to be rather subtle, and we need to restrict our attention to the case in which
the varieties are a finite union of irreducible varieties and a discrete variety. 
The isomorphism is just composition with this biholomorphism.

These methods also allow us to show that an isometric isomorphism 
is just composition with a conformal automorphism of the ball, and thus
is completely isometric and unitarily implemented.

Some counterexamples show that a biholomorphism between varieties does not
always yield an isomorphism of the multiplier algebras. 
We discuss a number of cases where we can establish the desired converse.

Arias and Latr\'emoli\`ere \cite{AriasLat10a} have an interesting paper in which 
they study certain operator algebras of this type in the
case where the variety is a countable discrete subset of the unit disc 
which is the orbit of a point under the action of a Fuchsian group. 
They establish results akin to ours
in the completely isometric case using rather different methods.

\section{Reproducing kernel Hilbert spaces associated to analytic varieties}

\subsection*{Basic notation}

Let  $H^2_d$ be Drury-Arveson space (see \cite{Arv98}). 
$H^2_d$ is the reproducing kernel Hilbert space on $\mb{B}_d$, 
the unit ball of $\mb{C}^d$, with kernel functions
\bes
k_\lambda (z) = \frac{1}{1-\lel z, \lambda\rir} \qfor z,\lambda \in \mb{B}_d.
\ees
We also consider the case $d = \infty$, and then $\mb{C}^d$ is understood as $\ell^2$.
We denote by $\cM_d$ the multiplier algebra $\textrm{Mult}(H^2_d)$ of $H^2_d$. 

Let $Z_1, \ldots, Z_d$ denote multiplication by the coordinate functions on $H^2_d$, 
given by 
\[
 (Z_i h) (z) = z_i h(z) \qfor i=1, \ldots, d.
\]
Let $\cA_d$ denote the norm closed algebra generated by $I, Z_1, \ldots, Z_d$. 
By \cite[Theorem 6.2]{Arv98}, $\cA_d$ is the universal (norm-closed) unital operator algebra generated by a commuting row contraction (see also \cite{Popescu99}).

We write $\cF = \cF(E)$ for the full Fock space
\[
\cF = \mb{C} \oplus E \oplus (E \otimes E) \oplus (E \otimes E \otimes E) \oplus \ldots ,
\] 
where $E$ is a $d$-dimensional Hilbert space.
Fix an orthonormal basis $\{e_1, \ldots, e_d\}$ for $E$. 
On $\cF$, we have the natural shift operators $L_1, \ldots, L_d$ given by 
\[
 L_j e_{i_1} \otimes \cdots \otimes e_{i_k} = 
 e_j \otimes e_{i_1} \otimes \cdots \otimes e_{i_k} 
 \qfor 1 \le j \le d.
\]
The \emph{non-commutative analytic Toeplitz algebra} $\cL_d$ is defined to be the 
unital \wot-closed algebra generated by $L_1, \ldots, L_d$. 

When $d$ is understood, we may write $\cA, \cM, \cL$ and $H^2$ 
instead of $\cA_d, \cM_d, \cL_d$ and $H^2_d$.

\subsection*{The RKHS of a variety}

For our purposes, an \emph{analytic variety} will be understood as 
the common zero set of a family of $H^2$ functions.
If $F$ is a subset of $H^2_d$, considered as functions on $\mb{B}_d$, let 
\[
 V(F) := \{ \lambda \in \mb{B}_d: f(\lambda) = 0 \FORAL f \in F \} .
\]
Propositions \ref{prop:zero_set} and \ref{prop:F_S_F_V} below both show that
there is not much loss of generality in taking this as our definition in this context. 
In particular, if $f\in \cM_d$, then $M_f 1 = f$ is a function in $H^2_d$.
So the zero set of a set of multipliers is an analytic variety.

Define 
\[
 J_V = \{f \in \cM : f(\lambda) = 0 \FORAL \lambda \in V  \} .
\]
Observe that $J_V$ is a \wot-closed ideal in $\cM$.

\begin{proposition} \label{prop:zero_set}
Let $F$ be a subset of $H^2$, and let $V = V(F)$. 
Then 
\[
 V = V(J_V) = \{ \lambda \in \mb{B}_d : f(\lambda) = 0 \FORAL f \in J_V  \}.
\]
\end{proposition}

\begin{proof}
Obviously $V \subseteq V(J_V)$. 
For the other inclusion, recall that \cite[Theorem 9.27]{AM02} states that a zero set 
of an $H^2$ function is a weak zero set for $\cM$ (i.e.\ the intersection of zero sets of functions 
in $\cM$). 
Since $V$ is the intersection of zero sets for $H^2$, it is a weak zero set for $\cM$; 
i.e., there exists a set $S \subseteq \cM$ such that $V = V(S)$. 
Now, $S \subseteq J_V$, so $V = V(S) \supseteq V(J_V)$.
\end{proof}

Given the analytic variety $V$, we define a subspace of $H^2_d$ by
\[
 \cF_V = \ol{\spn}\{k_\lambda : \lambda \in V\}.
\]
The Hilbert space $\cF_V$ is naturally a reproducing kernel Hilbert space of 
functions on the variety $V$.
One could also consider spaces of the form 
$\cF_S = \ol{\spn}\{k_\lambda : \lambda \in S\}$ 
where $S$ is an arbitrary subset of the ball. 
The following proposition shows that there is no loss of generality in considering 
only analytic varieties generated by $H^2$ functions. 

\begin{proposition}\label{prop:F_S_F_V}
Let $S \subseteq \mb{B}_d$.
Let $J_S$ denote the set of multipliers vanishing on $S$, and let $I_S$ 
denote the set of all $H^2$ functions that vanish on $S$. 
Then 
\[ \cF_S = \cF_{V(I_S)} = \cF_{V(J_S)} . \]
\end{proposition}

\begin{proof}
Clearly $\cF_S \subseteq \cF_{V(I_S)}$. 
Let $f \in \cF_S^\perp$. 
Then $f(x) = 0$ for all $x \in S$; so $f \in I_S$.
Hence by definition, $f(z) = 0$ for all $z \in V = V(I_S)$; 
whence $f \in \cF_{V(I_S)}^\perp$. 
Therefore $\cF_S = \cF_{V(I_S)}$.
The extension to zero sets of multipliers follows again from \cite[Theorem 9.27]{AM02}.
\end{proof}

\begin{remarks}
In general, it is not true that $V(I_S)$ is equal to the smallest analytic variety 
\textit{in the classical sense} containing $S \subseteq \mb{B}_d$. 
In fact, by Weierstrass's Factorization Theorem, 
every discrete set $Z = \{z_n\}_{n=1}^\infty$ 
in $\mb{D}$ is the zero set of some holomorphic function on $\mb{D}$. 
However, if the sequence $Z$ is not a Blaschke sequence, then there is no nonzero 
function in $H^2$ that vanishes on all of it. 
So here $I_Z = \{0\}$, and therefore $V(I_Z) = \mb{D}$. 

One very nice property of classical varieties is that the definition is local.
Because our functions must be multipliers, a strictly local definition does not
seem to be possible. However one could consider the following variant:
$V$ is a variety if for each point $\lambda\in\mb B_d$, there is an $\ep>0$
and a finite set $f_1,\dots,f_n$ in $\cM_d$ 
so that 
\[ b_\ep(\lambda) \cap V = \{ z \in b_\ep(\lambda) : 0 = f_1(z)= \dots = f_n(z) \} .\]
We do not know if every variety of this type is actually the intersection of zero sets.

In particular, we will say that a variety $V$ is {\em irreducible} if for any regular point $\lambda \in V$,
the intersection of zero sets of all multipliers vanishing on a small neighbourhood 
$V\cap b_\ep(\lambda)$ is exactly $V$. However we do not know whether an
irreducible variety is connected. A local definition of our varieties would presumably
clear up this issue. 
\end{remarks}

\subsection*{Ideals and invariant subspaces}
We will apply some results of Davidson-Pitts \cite[Theorem 2.1]{DavPitts2} and
\cite[Corollary 2.3]{DavPittsPick} to the commutative context.

In the first paper, a bijective correspondence is established between the collection of
\wot-closed ideals $J$ of $\cL_d$ and the complete lattice of subspaces which are 
invariant for both $\cL_d$ and its commutant $\cR_d$, the algebra of right multipliers.
The pairing is just the map taking an ideal $J$ to its closed range $\mu(J):=\ol{J \F}$.
The inverse map takes a subspace $N$ to the ideal $J$ of elements with range
contained in $N$.

In \cite[Theorem 2.1]{DavPittsPick}, it is shown that the quotient algebra
$\cL_d/J$ is completely isometrically isomorphic and \wot-homeomorphic 
to the compression of $\cL_d$ to $\mu(J)^\perp$.
In particular, \cite[Corollary 2.3]{DavPittsPick} shows that the multiplier algebra $\cM_d$
is completely isometrically isomorphic to $\cL_d/\cC$, where $\cC$ is the
\wot-closure of the commutator ideal of $\cL_d$. 
In particular, $\mu(\cC)^\perp = H^2_d$.

It is easy to see that there is a bijective correspondence between
the lattice of \wot-closed ideals $\operatorname{Id}(\cM_d)$ of $\cM_d$ and 
the \wot-closed ideals of $\cL_d$ which contain $\cC$.
Similarly there is a bijective correspondence between invariant subspaces $N$ of 
$\cM_d$ and invariant subspaces of $\cL_d$ which contain $\mu(\cC) = H_d^{2\perp}$.
Since the algebra $\cM_d$ is abelian, it is also the quotient of $\cR_d$ by its
commutator ideal, which also has range $H_d^{2\perp}$.
So the subspace $N \oplus H_d^{2\perp}$ is invariant for both $\cL_d$ and $\cR_d$.
Therefore an application of \cite[Theorem 2.1]{DavPitts2} yields the following
consequence:

\begin{theorem} \label{thm:DavPitts2}
Define the map $\alpha : \operatorname{Id}(\cM_d) \to \Lat (\cM_d)$ by 
$\alpha(J) = \ol{J 1}$.
Then $\alpha$ is a complete lattice isomorphism whose inverse $\beta$ is given by
\[ \beta (N) = \{ f \in \cM_d: f \cdot 1 \in N\}. \]
\end{theorem}

\noindent Moreover \cite[Theorem 2.1]{DavPittsPick} then yields:

\begin{theorem} \label{thm:DavPitts3}
If $J$ is a \wot-closed ideal of $\cM_d$ with range $N$, then
$\cM_d/J$ is completely isometrically isomorphic and \wot-homeomorphic
to the compression of $\cM_d$ to $N^\perp$.
\end{theorem}
\smallbreak

\subsection*{The multiplier algebra of a variety}\hspace*{-.7ex}
The reproducing kernel Hilbert space $\cF_V$ comes with its multiplier algebra 
$\cM_V = \Mult(\cF_V)$.
This is the algebra of all functions $f$ on $V$ such that $f h \in \cF_V$ 
for all $h \in \cF_V$. 
A standard argument shows that each multiplier determines a bounded
linear operator $M_f \in \B(\cF_V)$ given by $M_f h = f h$.
We will usually identify the function $f$ with its multiplication operator $M_f$.
We will also identify the subalgebra of $\B(\cF_V)$ consisting of the $M_f$'s 
and the algebra of functions $\cM_V$ (endowed with the same norm). 
One reason to distinguish $f$ and $M_f$ is that sometimes we need to consider 
the adjoints of the operators $M_f$. 
The distinguishing property of these adjoints is that 
$M_f^* k_\lambda = \ol{f(\lambda)}k_\lambda$ for $\lambda \in V$, in the sense that
if $A^* k_\lambda = \ol{f(\lambda)}k_\lambda$  for $\lambda \in V$,
then $f$ is a multiplier.

The space $\cF_V$ is therefore invariant for the adjoints of multipliers;
and hence it is the complement of an invariant subspace of $\M$.
Thus an application of Theorem~\ref{thm:DavPitts3} and the
complete Nevanlinna-Pick property yields:

\begin{proposition}\label{prop:complete_quotient}
Let $V$ be an analytic variety in $\bB_d$. Then
\[
 \cM_V = \{f |_{V} : f \in \cM \}.
\]
Moreover the mapping $\phi : \cM \rightarrow \cM_V$ given by 
$\phi(f) = f|_V$ induces a completely isometric isomorphism 
and \wot-homeomorphism of $\cM / J_V$ onto $\cM_V$.
For any $g \in \cM_V$ and any $f\in\cM$ such that $f|_V=g$,
we have $M_g = P_{\cF_V} M_f|_{\cF_V}$.
Given any $F \in M_k(\cM_V)$, one can choose $\widetilde F \in M_k(\cM)$
so that $\widetilde F|_V = F$ and $\|\widetilde F\|=\|F\|$.
\end{proposition}

\begin{proof}
Theorem~\ref{thm:DavPitts3} provides the isomorphism between $\cM / J_V$
and the restriction of the multipliers to $N^\perp$ where $N = \ol{J_V 1}$.
Since $J_V$ vanishes on $V$, if $f\in J_V$, we have 
\[
 \ip{M_f h,k_\lambda} = \ip{h, M_f^* k_\lambda} = 0 
 \qforal \lambda\in V \AND h \in H^2_d .
\]
So $N$ is orthogonal to $\cF_V$. 
Conversely, if $M_f$ has range orthogonal to $\cF_V$,
the same calculation shows that $f\in J_V$. 
Since the pairing between subspaces and ideals is bijective, we deduce that
$N=\cF_V^\perp$.
The mapping of $\cM / J_V$ into $\cM_V$ is given by compression to $\cF_V$ 
by sending $f$ to $P_{\cF_V} M_f|_{\cF_V}$.

It is now evident that the restriction of a multiplier $f$ in $\cM$ to $V$ yields
a multiplier on $\cF_V$, and that the norm is just $\|f+ J_V\| = \|P_{\cF_V} M_f|_{\cF_V}\|$.
We need to show that this map is surjective and completely isometric.
This follows from the complete Nevanlinna-Pick property as in 
\cite[Corollary 2.3]{DavPittsPick}.
Indeed, if $F \in M_k(\cM_V)$ with $\|F\|=1$, then standard computations 
show that if $\lambda_1,\dots,\lambda_n$ lie in $V$, then
\[
\Big[\big(I_k - F(\lambda_j) F(\lambda_i)^* \big) \ip{ k_{\lambda_i}, k_{\lambda_j}}  \Big]_{n\times n}  
\]
is positive semidefinite. 
By \cite{DavPittsPick}, this implies that there is a matrix multiplier 
$\widetilde{F} \in M_k(\cM)$ with $\|\widetilde{F}\| = 1$ such that 
$\widetilde{F} |_V = F$.
\end{proof}

We can argue as in the previous subsection that there is a bijective correspondence
between \wot-closed ideals of $\cM_V$ and its invariant subspaces:

\begin{corollary} \label{cor:DavPitts1}
Define the map $\alpha : \operatorname{Id}(\cM_V) \rightarrow \Lat (\cM_V)$ by $\alpha(J) = \ol{J 1}$.   Then $\alpha$ is a complete lattice isomorphism whose inverse $\beta$ is given by
\[\beta (N) = \{ f \in \cM_V: f\cdot 1 \in N\}. \]
\end{corollary}

\begin{remark}
By Theorem 4.2 in \cite{AM00}, every irreducible complete Nevan\-linna-Pick kernel 
is equivalent to the restriction of the kernel of Drury-Arveson space to a subset of the ball. 
It follows from this and from the above discussion that every multiplier algebra of an
irreducible complete Nevanlinna-Pick kernel is completely isometrically isomorphic
to one of the algebras $\cM_V$ that we are considering here.
\end{remark}

\begin{remark}
By the universality of $Z_1, \ldots, Z_d$  \cite{Arv98}, for every 
unital operator algebra $\mathcal{B}$ that is generated by a \emph{pure} 
commuting row contraction $T = (T_1, \ldots, T_d)$, there exists a unital 
homomorphism $\phi_T: \cM \rightarrow \mathcal{B}$ that gives rise to a 
natural functional calculus 
\[ f(T_1, \ldots, T_d) = \phi_T(f)  \qfor f \in \cM . \]
So it makes sense to say that a commuting row contraction $T$ annihilates 
$J_V$ if $\phi_T$ vanishes on $J_V$.  
By Proposition \ref{prop:complete_quotient}, we may identify $\cM_V$ with 
the quotient $\cM/J_V$, thus we may identify $\cM_V$ as the universal 
\wot-closed unital operator algebra generated by a \emph{pure} commuting 
row contraction $T = (T_1, \ldots, T_d)$ that annihilates $J_V$. 
\end{remark}

\section{The character space of $\cM_V$}\label{sec:characters}

If $A$ is a Banach algebra, denote the set of multiplicative linear functionals on $A$ 
by $M(A)$; and endow this space with the weak-$*$ topology.
We refer to elements of $M(A)$ as \emph{characters}. 
Note that all characters are automatically unital and continuous with norm one. 
When $A$ is an operator algebra, characters are automatically completely contractive.

When $V$ is an analytic variety in $\mb{B}_d$, we will abuse notation and 
let $Z_1, \ldots, Z_d$ also denote the images of the coordinate functions 
$Z_1, \ldots, Z_d$ of $\cM$ in $\cM_V$.
Since $\big[ Z_1,\ \ldots,\ Z_d \big]$ is a row contraction, 
\[ \big\| \big(\rho(Z_1), \ldots, \rho(Z_d) \big)\big\| \le 1 \qforal \rho \in M(\cM_V) .\]
The map $\pi : M(\cM_V) \rightarrow \ol{\mb{B}}_d$ given by 
\[
 \pi(\rho) = (\rho(Z_1), \ldots, \rho(Z_d))
\]
is continuous as a map from $M(\cM_V)$, with the weak-$*$ topology, 
into $\ol{\mb{B}}_d$ (endowed with the weak topology in the case $d= \infty$). 
We define 
\[ \ol{V}^{\cM} = \pi((M(\cM_V)) . \]

Since $\pi$ is continuous, $\ol{V}^{\cM}$ is a (weakly) compact subset of $\ol{\mb{B}_d}$.
For every $\lambda \in \ol{V}^{\cM}$, the {\em fiber} over $\lambda$ is defined 
to be the set $\pi^{-1}(\lambda)$ in $M(\cM_V)$. 
We will see below that $V \subseteq \ol{V}^{\cM}$, 
and that when $d<\infty$ the fiber over every $\lambda \in V$ is a singleton.

Every unital homomorphism $\phi : A \rightarrow B$ between Banach algebras
induces a mapping $\phi^* : M(B) \rightarrow M(A)$ by $\phi^* \rho = \rho \circ \phi$. 
If $\phi$ is an isomorphism, then $\phi^*$ is a homeomorphism. 
We will see below that in many cases the map 
$\phi^* : M(\cM_W) \rightarrow M(\cM_V)$ has additional structure, 
the most important aspect being that $\phi^*$ induces a holomorphic map 
from $W$ into $V$. 

\subsection*{The weak-$*$ continuous characters of $\cM_V$}

In the case of $\cM_d$, the weak-$*$ continuous characters coincide with
the point evaluations at points in the open ball \cite{AriasPopescu95, DavPitts1}
\[
 \rho_\lambda(f) = f(\lambda) = \ip{ f \nu_\lambda, \nu_\lambda} 
 \qfor \lambda \in \mb B_d,
\]
where $\nu_\lambda = k_\lambda/\|k_\lambda\|$. 
The fibers over points in the boundary sphere are at least as complicated as the
fibers in $M(H^\infty)$ \cite{DavPitts2}, which are known to be extremely large \cite{Hoffman}.

As a quotient of a dual algebra by a weak-$*$ closed ideal, the algebra $\cM_V$ 
inherits a weak-$*$ topology. 
As an operator algebra concretely represented on a reproducing kernel Hilbert space, 
$\cM_V$ also has the weak-operator topology (\wot). 
In \cite[Lemma 11.9]{DRS} we observed that, as is the case for the free semigroup 
algebras $\cL_d$ \cite{DavPitts1}, the weak-operator and weak-$*$ topologies on
$\cM_V$ coincide. 
The setting there was slightly different, but the proof remains the same.
It relies on the observation \cite{AriasPopescu00} that $\cM_V$ 
has property $\mb A_1(1)$.

\begin{lemma} \label{lem:weak_weakstar}
The weak-$*$ and weak-operator topologies on $\cM_V$ coincide.
\end{lemma}

\begin{proposition}\label{prop:charM_V}
The \wot-continuous characters of $\cM_V$ can be identified with $V$ via the correspondence $\lambda \leftrightarrow \rho_\lambda$. 
If $d<\infty$, then $\ol{V}^{\cM} \cap \mb{B}_d = V$, and for for every $\lambda \in V$ the fiber $\pi^{-1}(\lambda)$ is a singleton. 
\end{proposition}

\begin{proof}
As $\cM_V$ is the multiplier algebra of a reproducing kernel Hilbert space on $V$, 
it is clear that for each $\lambda \in V$, the evaluation functional 
\[ \rho_\lambda(f) = f(\lambda) = \ip{ f \nu_\lambda, \nu_\lambda} \] 
is a \wot-continuous character.

On the other hand, the quotient map from the free semigroup algebra $\cL$ 
onto $\cM_V$ is weak-operator continuous. 
Thus, if $\rho$ is a \wot-continuous character of $\cM_V$, then it induces
a \wot-continuous character on $\cL$ by composition. 
Therefore, using \cite[Theorem 2.3]{DavPitts2}, we find that $\rho$ must be equal 
to the evaluation functional $\rho_\lambda$ at some point $\lambda \in \mb{B}_d$. 
Moreover $\rho_\lambda$ annihilates $J_V$.
By Proposition \ref{prop:zero_set}, the point $\lambda$ lies in $V$. Thus we have an identification of the \wot-continuous characters with $V$. 

Now assume that $d<\infty$. If $\rho$ is a character on $\cM_V$ such that $\pi(\rho) = \lambda \in \mb B_d$,
then again it induces a character $\tilde\rho$ on $\cL$ with the property that
$\tilde\rho(L_1,\dots,L_d) = \lambda$. 
Since $d<\infty$, we have
by \cite[Theorem 3.3]{DavPitts2}, that $\tilde\rho$ is \wot-continuous
and coincides with point evaluation.
Hence by the previous paragraph, $\lambda$ belongs to $V$, so $\rho = \rho_\lambda$.
This provides the identification of $\pi^{-1}(V)$ with $V$ in the case $d<\infty$. 
\end{proof}

Thus the character space $M(\cM_V)$ consists of $V$ and $M(\cM_V) \setminus V$, which we call the \emph{corona}. 
By definition, the corona is fibered over $\ol{V}^{\cM} \setminus V$, 
and by the above proposition, when $d<\infty$ this latter set is contained in $\partial \mb{B}_d$.

\begin{remark}
When $d = \infty$, it may happen that $\pi^{-1}(\lambda)$ has more than one element, and it may also happen that $\ol{V}^{\cM} \cap \mb{B}_d \supsetneq V$. Such examples were found by Michael Hartz 
(private communication).
\end{remark}

\subsection*{Isomorphisms and induced maps}

\begin{proposition}\label{prop:induced_map}
Let $V$ and $W$ be varieties in $\B_d$, and let $\phi : \M_V \rightarrow \M_W$ be a unital homomorphism. 
Then $\phi$ gives rise to function $F_\phi : W \rightarrow \ol{\mb{B}_d}$ by 
\[
F_\phi = \pi \circ \phi^*\big|_W. 
\]
Moreover, there exist multipliers $F_1, F_2, \ldots \in \M_\infty$ such that $F_\phi = (F_1\big|_W, \, F_2\big|_W,\, \ldots)$. For every $\lambda \in W$, $\phi^*(\rho_\lambda)$ lies in the fiber over $F_\phi(\lambda)$. If $\phi$ is completely bounded or $d<\infty$, then $F_\phi$ extends to a holomorphic function defined on $\bB_d$. 
\end{proposition}

\begin{proof}
By Proposition \ref{prop:charM_V}, we may identify $W$ as a subset of $M(\M_W)$, so the definition $F_\phi = \pi \circ \phi^*\big|_W$ makes sense. 
Now for every $i=1,2,\ldots$, we compute for every $\lambda \in W$
\[
\phi^* (\rho_\lambda)(Z_i) = \phi(Z_i) (\lambda). 
\]
Now, $\phi(Z_i) \in \M_W$ for all $i$. 
By Proposition \ref{prop:complete_quotient}, for every $i$ there is an element $F_i \in \M$ such that $\phi(Z_i) = F_i \big|_W$ and $\|F_i\| = \|\phi(Z_i)\|$. 
Moreover, the norm of the row operator $\big [ M_{F_1},\ \ldots,\ M_{F_d} \big]$ can be chosen equal to the multiplier norm of $\big[ \phi(Z_1),\ \ldots,\ \phi(Z_d) \big]$, in case the latter is finite. 
In any case we find that for all $\lambda \in W$, 
\[
F_\phi(\lambda) = (F_1(\lambda), F_2(\lambda), \ldots ). \qedhere
\]
\end{proof}

\begin{proposition}\label{prop:weak_star_homo}
Let $V$ and $W$ be varieties in $\bB_d$. If $\phi : \M_V \rightarrow \M_W$ is a unital homomorphism such that $\phi^*$ takes weak-$*$ continuous characters to weak-$*$ continuous characters, then $F_\phi(W) \subseteq V$ and $\phi$ is given by 
\[
\phi(f) = f \circ F_\phi \quad , \quad f \in \M_V. 
\]
\end{proposition}
\begin{proof}
Since $\phi^*$ takes weak-$*$ continuous functionals to weak-$*$ continuous functional, $\phi^*(\rho_\lambda)$ is an evaluation functional. 
Since it lies in the fiber $\pi^{-1}(F(\lambda))$ by definition, then $\phi^*(\rho_\lambda) = \rho_{F_\phi(\lambda)}$ and $F_\phi(\lambda) \in V$ for every $\lambda \in W$. 

Let $f \in \M_V$ and let $\lambda \in W$. Then 
\[
\phi(f)(\lambda) = \rho_\lambda (\phi(f)) = \phi^*(\rho_\lambda)(f) = f(F_\phi(\lambda)).
\]
This shows that $\phi$ is implemented by composition with $F_\phi$. 
\end{proof}

\begin{corollary}\label{cor:weak_star_iso}
Let $V$ and $W$ be varieties in $\bB_d$. If $\phi : \M_V \rightarrow \M_W$ is a unital  homomorphism that is weak-$*$ continuous, then  $F_\phi(W) \subseteq V$ and $\phi$ is given by 
\[
\phi(f) = f \circ F_\phi \quad , \quad f \in \M_V. 
\]
Moreover, if there there exists a weak-$*$ continuous isomorphism $\phi : \M_V \rightarrow \M_W$, then $F_\phi(W) = V$, $F_{\phi^{-1}}(V) = W$, and there are sequences of multipliers $F_1, F_2, \ldots \in \M$ and $G_1, G_2, \ldots \in \M$ such that 
\[
F_\phi = (F_1\big|_W,\, F_2\big|_W,\, \ldots ), 
\]
and 
\[
F_{\phi^{-1}} = (G_1\big|_V,\, G_2\big|_V,\, \ldots ). 
\]
\end{corollary}

\begin{proof}
The first follows from Proposition \ref{prop:weak_star_homo}. The second part follows easily once we note that the inverse of a weak-$*$ continuous isomorphism is weak-$*$ continuous. 
\end{proof}

\begin{remark}
In the case $d<\infty$, since the maps $F_i$ and $G_i$ above are defined on the whole unit ball $\bB_d$, the maps $F_\phi$ and $F_{\phi^{-1}}$ extend to the ball, and we find that $V$ and $W$ are biholomorphic. Since the maps involved are multipliers --- not merely bounded holomorphic functions --- then we say that $V$ and $W$ are {\em multiplier biholomorphic}. 
When $d=\infty$ we do not say that $V$ and $W$ are biholomorphic because for $\lambda \in \bB_\infty \setminus W$, we do not know if $(F_1(\lambda), F_2(\lambda), \ldots, )$ is in $\ell^2$. 
\end{remark}

\begin{corollary}\label{cor:weak_star_iso}
Let $V$ and $W$ be varieties in $\bB_d$ for $d<\infty$. If there there exists a weak-$*$ continuous isomorphism $\phi : \M_V \rightarrow \M_W$, then $V$ and $W$ are multiplier biholomorphic. 
\end{corollary}

In the next sections we will try to control $F_\phi$ given assumptions on the variety or the isomorphisms involved, in order to obtain a classification. Without the assumption of finite dimensionality or weak-$*$ continuity, the situation becomes significantly more involved.

\section{Completely isometric isomorphisms }\label{S:cc_iso}

The (completely) isometric automorphisms of $\cM$ arise as composition with an automorphism of the ball 
(i.e., a biholomorphism of the ball onto itself). 
This can be deduced from \cite[Section 4]{DavPitts2}, or alternatively from 
Theorems 3.5 and 3.10 in \cite{Popescu10}. 
In \cite[Section 9]{DRS}, we wrote down the explicit form of the unitaries 
on $H^2$ that implement these automorphisms. 
We will use these unitaries to construct unitarily implemented, completely isometric 
isomorphisms of the multiplier algebras that we are studying. 
In addition, we will show that all completely isometric isomorphisms 
of these algebras arise in this way. 

\begin{proposition} \label{prop:auto_implemt_iso}
Let $V$ and $W$ be varieties in $\mb{B}_d$.
Let $F$ be an automorphism of $\mb{B}_d$ that maps $W$ onto $V$. 
Then $f \mapsto f \circ F$ is a unitarily implemented completely 
isometric isomorphism of $\cM_V$ onto $\cM_W$;
i.e. $M_{f\circ F} = U M_f U^*$. 
The unitary $U^*$ is the linear extension of the map 
\[ U^* k_w = c_w k_{F(w)} \qfor  w \in W ,  \] 
where $c_w = (1-\|F^{-1}(0)\|^2)^{1/2} \ol{  k_{F^{-1}(0)} (w)}$.
\end{proposition}

\begin{proof}
Let $F$ be such an automorphism, and set $\alpha = F^{-1}(0)$. 
By \cite[Theorem 9.2]{DRS}, the unitary map $U \in \B(H^2)$ is given by
\[
 U h = (1-\|\alpha\|^2)^{1/2} k_{\alpha} (h \circ F) \qfor h \in H^2 .
\]
As $F(W) = V$, $U$ takes the functions in $H^2$ 
that vanish on $V$ to the functions in $H^2$ that vanish on $W$. 
Therefore it takes $\cF_V$ onto $\cF_W$. 

Let us compute $U^*$. 
For $h \in H^2$ and $w \in W$, we have
\begin{align*}
 \lel h, U^* k_w \rir &= \lel U h , k_w \rir \\
 &= \lel (1-\|\alpha\|^2)^{1/2} k_{\alpha} (h \circ F) , k_w \rir \\ 
 &= (1-\|\alpha\|^2)^{1/2} k_{\alpha} (w)\, h (F(w)) \\
 &=  \lel h , c_w k_{F(w)} \rir ,
\end{align*}
where $c_w = (1-\|F^{-1}(0)\|^2)^{1/2} \ol{  k_{F^{-1}(0)} (w)}$. 
Thus $U^* k_w = c_w k_{F(w)}$. 
Note that since $U^*$ is a unitary, $|c_w| = \|k_w\|/\|k_{F(w)}\|$. 

Finally, we show that conjugation by $U$ implements the isomorphism 
between $\cM_V$ and $\cM_W$ given by composition with $F$.
Observe that $U c_w k_{F(w)} = k_w$. 
For $f \in \cM_V$ and $w \in W$, 
\[
 U M_f^* U^* k_w =  U M_f^* c_w k_{F(w)} 
 = \ol{f(F(w))} U c_w k_{F(w))} =  \ol{(f\circ F)(w)} k_w .
\]
Therefore $f\circ F$  is a multiplier on $\cF_W$ and $M_{f\circ F} = U M_f U^*$.
\end{proof}

Now we turn to the converse.

\begin{lemma}\label{L:cc_homo_holo}
Let $V \subseteq \mb{B}_d$ and $W \subseteq \mb{B}_{d'}$ be varieties. 
Let $\phi$ be a unital, completely contractive algebra isomorphism of
$\cM_V$ into $\cM_W$. If $d=\infty$ then assume also that $\phi$ is isometric. 
Then there exists a holomorphic map $F: \mb{B}_{d'} \to \mb{B}_d$ 
such that
\begin{enumerate}
\item $F(W) \subseteq V$.
\item $F|_{W} = \phi^*|_{W}$.
\item the components $F_1, \ldots, F_d$ of $F$ form a row contraction
 of operators in $\cM_{d'}$.
\item $\phi$ is given by composition with $F$, that is
\[
 \phi(f) = f \circ F \qfor f \in \cM_V .
\]
\end{enumerate}
\end{lemma}

\begin{proof}
We define $F = F_\phi$ to be the row contractive multiplier 
$F_\phi = \big[ F_1, \ldots, F_d \big]$ with coefficients in $\cM_{d'}$ as in Proposition \ref{prop:induced_map}.
As $F$ is contractive as a multiplier, it is also contractive in the sup norm.
Moreover, since $\phi$ is injective, the $F_i$ and $F$ are non-constant holomorphic functions. 
Therefore $F$ must have range in the open ball $\mb{B}_d$. 

By Proposition \ref{prop:weak_star_homo} it remains to show that $\phi^*$ takes weak-$*$ continuous characters to weak-$*$ continuous characters. 
When $d < \infty$ this follows from Propositions \ref{prop:charM_V} and \ref{prop:induced_map} and the fact noted above that $F$ takes values in $\bB_d$. 

If $d=\infty$ we invoke \cite[Corollary 6.6]{KenYan12}, which states that $\M_V$ and $\M_W$ both have strongly unique preduals. It then follows that if $\phi$ is an isometric isomorphism between theses algebras, then $\phi$ is also a weak-$*$ homeomorphism. Therefore $\phi^*$ takes weak-$*$ continuous characters to weak-$*$ continuous characters. 
\end{proof}

\begin{rem}\label{R:cartan}
For the following lemma, we need some results of Cartan and Rudin on bounded domains in $\mb{C}^n$
for the case of $\mb{B}_\infty$.

Cartan's uniqueness theorem \cite[Theorem~2.1.1]{RudinBall} says that a holomorphic function
$F:\Omega\to\Omega$, where $\Omega$ is a bounded domain in $\mb{C}^n$, that satisfies
$F(p)=p$ and $F'(p)=I$ is the identity function $F(z)=z$. 
This remains valid for bounded domains in $\ell^2$ using the same proof.
This yields the corollary \cite[Theorem~2.1.3]{RudinBall} that if $\Omega$ is a bounded circular
domain containing $0$, and $F$ is a biholomorphism of $\Omega$ such that $F(0)=0$,
then $F$ is linear. This only relies on the previous result, so is also valid for $\mb{B}_\infty$.

We also need a result of Rudin \cite[Theorem~8.2.3]{RudinBall} that if $F:\mb{B}_d \to \mb{B}_d$
is holomorphic, then the fixed point set is affine. This is deduced from a result that if $\Omega$ is a balanced, bounded, strictly convex domain in $\mb{C}^n$, and $F:\Omega\to\Omega$ is holomorphic
and $F(0)=0$, then the fixed points of $F$ coincide with the fixed points of $F'(0)$.
The main idea is that $\Omega$ is the unit ball of a Banach space.
So again, these results are also valid for $\mb{B}_\infty$.
\end{rem}

\begin{lemma}\label{L:comp_iso_iso_biholo}
Let $0 \in V \subseteq \mb{B}_d$ and $0 \in W \subseteq \mb{B}_{d'}$ be varieties. 
Let $\phi : \cM_V \rightarrow \cM_W$ be a completely isometric isomorphism
such that $\phi^*\rho_0 = \rho_0$. 
Then there exists an isometric linear map $F$ of $\mb{B}_{d'}\cap \spn{W}$ 
onto $\mb{B}_d \cap \spn{V}$ such that $F(W) = V$, $F(0)=0$
and $F|_W = \phi^*$.
\end{lemma}

\begin{proof}
By making $d$ smaller, we may assume that $\mb{C}^d = \spn{V}$. 
Similarly, we may assume $\mb{C}^{d'} = \spn{W}$.

By Lemma \ref{L:cc_homo_holo} applied to $\phi$, there is a holomorphic map 
$F$ of $\mb{B}_{d'}$ into $\mb{B}_d$ that implements $\phi^*$.
Thus $F(W) \subseteq V$ and $F(0)=0$. 
By the same lemma applied to $\phi^{-1}$, there is a holomorphic map $G$
of $\mb{B}_d$ into $\mb{B}_{d'}$ that implements $(\phi^{-1})^*$.
Hence $G(V) \subseteq W$ and $G(0)=0$. 
Now, $\phi^*$ and $(\phi^{-1})^*$ are inverses of each other.
Therefore $F \circ G |_V$ and $G \circ F |_W$ are the identity maps.

Let $H = F \circ G$. 
Then $H$ is a holomorphic map of $\mb{B}_d$ into itself, such that 
$H|_V$ is the identity. In particular $H(0) = 0$. 
By \cite[Theorem 8.2.2]{RudinBall} and Remark~\ref{R:cartan}, 
the fixed point set of $H$ is an affine set equal to 
the fixed point set of $H'(0)$ in $\mb{B}_d$.
Therefore $H$ is the identity on $\mb{B}_d$ since $\mb{C}^d = \spn{V}$. 
Applying the same reasoning to $G \circ F$, we see that $F$ is a biholomorphism of 
$\mb{B}_{d'}$ onto $\mb{B}_d$ such that $F(W) = V$.
In particular, $d'=d$.
It now follows from a theorem of Cartan \cite[Theorem~2.1.3]{RudinBall} and Remark~\ref{R:cartan}
that $F$ is a unitary linear map. 
\end{proof}

Now we combine these lemmas to obtain the main result of this section.

If $V \subseteq \mb{B}_d$ and $W \subseteq \mb{B}_{d'}$ are varieties, 
then we can consider them both as varieties in $\mb{B}_{\max(d,d')}$. 
Therefore, we may assume that $d=d'$. 
This does not change the operator algebras.  See \cite[Remark 8.1]{DRS}.

\begin{theorem}\label{theorem:complete_iso_iff_auto}
Let $V$ and $W$ be varieties in $\mb{B}_d$. 
Then $\cM_V$ is completely isometrically isomorphic to $\cM_W$ if and only 
if there exists an automorphism $F$ of $\mb{B}_d$ such that $F(W) = V$. 

In fact, every completely isometric isomorphism $\phi: \cM_V \rightarrow \cM_W$ 
arises as composition $\phi(f) = f \circ F$ where $F$ is such an automorphism. 
In this case, $\phi$ is unitarily implemented by the unitary sending the 
kernel function $k_w \in \cF_W$ to a scalar multiple of the kernel function 
$k_{F(w)} \in \cF_V$.
\end{theorem}

\begin{proof}
If there is such an automorphism, then the two algebras are completely
isometrically isomorphic by Proposition~\ref{prop:auto_implemt_iso};
and the unitary is given explicitly there.

Conversely, assume that $\phi$ is a completely isometric isomorphism 
of $\cM_V$ onto $\cM_W$.
By Lemma~\ref{L:cc_homo_holo}, $\phi^*$ maps $W$ into $V$.
Pick a point $w_0 \in W$ and set $v_0 = \phi^*(w_0)$.
By applying automorphisms of $\mb{B}_d$ that move
$v_0$ and $w_0$ to $0$ respectively,
and applying Proposition~\ref{prop:auto_implemt_iso}, 
we may assume that $0 \in V$ and $0 \in W$ and $\phi^*(0)=0$.

Now we apply Lemma~\ref{L:comp_iso_iso_biholo} to obtain an 
isometric linear map $F$ of the ball $\mb{B}_d\cap \spn{W}$  
onto the ball $\mb{B}_d \cap \spn{V}$ such that $F|_W = \phi^*$.
In particular, $\spn W$ and $\spn V$ have the same dimension. 
We may extend the definition of $F$ to a unitary map on $\mb C^d$,
and so it extends to a biholomorphism of $\mb{B}_d$.

Now Proposition~\ref{prop:auto_implemt_iso} yields a unitary which implements
composition by $\phi^*$.  
By Lemma~\ref{L:cc_homo_holo}, every completely isometric isomorphism $\phi$
is given as a composition by $\phi^*$.
So all maps have the form described.
\end{proof}

There is a converse to Lemma \ref{L:cc_homo_holo}, which may provide 
an alternative proof for one half of Theorem \ref{theorem:complete_iso_iff_auto}. 
Arguments like the following are not uncommon in the theory of RKHS; 
see for example \cite[Theorem 5]{Jury07}.

\begin{proposition}\label{prop:row_holo_cc}
Let $V \subseteq \mb{B}_d$ and $W \subseteq \mb{B}_{d'}$ be varieties. 
Suppose that there exists a holomorphic map $F: \mb{B}_{d'} \rightarrow \mb{B}_d$ 
that satisfies $F(W) \subseteq V$, such that the components 
$f_1, \ldots, f_d$ of $F$ form a row contraction of operators in $\cM_{d'}$.
Then the map given by composition with $F$
\[
 \phi(f) = f \circ F \qfor f \in \cM_V 
\]
yields a unital, completely contractive algebra homomorphism of $\cM_V$ into $\cM_W$. 
\end{proposition}

\begin{proof}
Composition obviously gives rise to a unital homomorphism, so all we have 
to demonstrate is that $\phi$ is completely contractive.
We make use of the complete NP property of these kernels.
 
Let $G \in M_k(\cM_V)$ with $\|G\|\leq 1$. 
Then for any $N$ points $w_1$, \dots, $w_N$ in $W$, we get $N$ points
$F(w_1)$, \dots, $F(w_N)$ in $V$.
The fact that $\|G\|\le 1$
implies that the $N \times N$ matrix with $k\times k$ matrix entries
\[
 \begin{bmatrix}
 \dfrac{I_k - (G\circ F)(w_i) (G\circ F)(w_j)^*}{1-\ip{F(w_i), F(w_j)}}
 \end{bmatrix}_{N\times N}
 \ge 0.
\]
Also, since $\|F\|\le1$ as a multiplier on $\cF_W$, we have that
\[
 \begin{bmatrix}
 \dfrac{1-\ip{F(w_i), F(w_j)}}{1-\ip{w_i,w_j}}
 \end{bmatrix}_{N\times N}
 \ge 0.
\]
Therefore the Schur product of these two positive matrices is positive:
\[
 \begin{bmatrix}
 \dfrac{I_k - (G\circ F)(w_i) (G\circ F)(w_j)^*}{1-\ip{w_i,w_j}}
 \end{bmatrix}_{N\times N}
 \ge 0.
\]
Now the complete NP property yields that $G\circ F$ is a contractive multiplier
in $M_k(\cM_W)$.
\end{proof}

%

\section{Isomorphisms of algebras and biholomorphisms} %
\label{sec:isomorphisms_biholomorphisms}

We turn now to the question: \emph{when does there exist an $($algebraic$)$ 
isomorphism between $\cM_V$ and $\cM_W$?} 
This problem is more subtle, and we frequently need to assume that 
the variety sits inside a finite dimensional ambient space. 
Even the construction of the biholomorphism seems to rely
on some delicate facts about complex varieties.

We begin with a well-known automatic continuity result.
Recall that a commutative Banach algebra is \emph{semi-simple} 
if the Gelfand transform is injective.

\begin{lemma}\label{lem:homo_is_cont}
Let $V$ and $W$ be varieties in $\mb{B}_d$. 
Every homomorphism from $\cM_V$ to $\cM_W$ is norm continuous.
\end{lemma}

\begin{proof}
The algebras that we are considering are easily seen to be semi-simple.
A general result in the theory of commutative Banach algebras says that 
every homomorphism into a semi-simple algebra is automatically continuous 
(see \cite[Prop.~4.2]{Dales}).
\end{proof}

\begin{lemma} \label{L:weak_algiso_biholo_discrete}
Let $V$ and $W$ be varieties in $\mb{B}_d$ and $\mb{B}_{d'}$, 
respectively, with $d'<\infty$. 
Let $\phi : \cM_V \rightarrow \cM_W$ be an algebra isomorphism. 
Suppose that $\lambda$ is an isolated point in $W$. Then $\phi^*(\rho_\lambda)$ is an 
evaluation functional at a point in $V$.
\end{lemma}

\begin{proof}
The character $\rho_\lambda$ is an isolated point in $M(\cM_W)$. 
(Here is where we need $d'<\infty$). 
Since $\phi^*$ is a homeomorphism, $\phi^*(\rho_\lambda)$ must also be an 
isolated point in $M(\cM_V)$. 
By Shilov's idempotent theorem (see \cite[Theorem 21.5]{BonsallDuncan}), 
the characteristic function $\upchi_{\phi^*(\rho_\lambda)}$ of $\phi^*(\rho_\lambda)$
belongs to $\cM_V$. 
Now suppose that $\phi^*(\rho_\lambda)$ is in the corona $M(\cM_V) \setminus V$. 
Then $\upchi_{\phi^*(\rho_\lambda)}$ vanishes on $V$.
Therefore, as an element of a multiplier algebra, this means that 
$\upchi_{\phi^*(\rho_\lambda)} = 0$. 
Therefore $\upchi_{\phi^*(\rho_\lambda)}$ must vanish on the entire maximal ideal 
space, which is a contradiction. 
Thus $\phi^*(\rho_\lambda)$ lies in $V$.
\end{proof}

Next we want to show that any algebra isomorphism $\phi$ between $\cM_V$ and $\cM_W$
must induce a biholomorphism between $W$ and $V$. This identification will be the
restriction of $\phi^*$ to the characters of evaluation at points of $W$. 
In order to achieve this, we need to make some additional assumptions.

Our difficulty is basically that we do not have enough information about varieties.
In the classical case, if one takes a regular point $\lambda \in V$, 
takes the connected component of $\lambda$ in the set of all regular points of $V$, 
and closes it up (in $\mb{B}_d$), then one obtains a subvariety. Moreover
the closure of the complement of this component is also a variety 
\cite[ch.3, Theorem~1G]{Whitney}.

However our varieties are the intersections of zero sets of a family of multipliers.
Let us say that a variety $V$ is \textit{irreducible} if for any regular point $\lambda \in V$,
the intersection of zero sets of all multipliers vanishing on a small neighbourhood 
$V\cap b_\ep(\lambda)$ is exactly $V$.
We do not know, for example, whether an irreducible variety in our sense is connected.
Nor do we know that if we take an irreducible subvariety of a variety, 
then there is a complementary subvariety as in the classical case.

A variety $V$ is said to be {\em discrete} if it has no accumulation points in $\mb{B}_d$. 

We will resolve the isomorphism problem in two situations. 
The first is the case of a finite union of
irreducible varieties and a discrete variety. 
The second is the case of an isometric isomorphism.
In the latter case, the isomorphism will turn out to be completely isometric.
This yields a different approach to the results of the previous section. 
In either case we will need to assume that $d<\infty$. 

We need some information about the maximal ideal space $M(\cM_V)$.
Recall that there is a canonical projection $\pi$ into $\ol{\mb{B}}_d$ obtained by
evaluation at $[Z_1,\dots,Z_d]$. For any point $\mu$ in the unit sphere,
$\pi^{-1}(\mu)$ is the fiber of $M(\cM_V)$ over $\mu$.
We saw in Proposition~\ref{prop:charM_V} that when $d<\infty$, for $\lambda\in\mb{B}_d$,
$\pi^{-1}(\lambda)$ is the singleton $\{\rho_\lambda\}$, the point evaluation at $\lambda$. 

By Proposition \ref{prop:weak_star_homo}, in order to show that an isomorphism is given by composition with a biholomorphism, we need to show that $\phi^*$ takes weak-$*$ continuous characters to weak-$*$ continuous characters. Since fibers over internal points of the ball are singletons, it suffices to show that $F_\phi$ takes values in the open ball. Since $\pi$ is injective on $\pi^{-1}(V)$ (when $d<\infty$) we may identify $\phi^*\big|_W$ with $F_\phi$. 

The following lemma is analogous to results about Gleason parts for function algebras.
However part (2) shows that this is different from Gleason parts, as 
disjoint subvarieties of $V$ will be at a distance of less than 2 apart.
This is because $\cM_V$ is a (complete) quotient of $\cM_d$, and thus
the difference $\|\rho_\lambda-\rho_\mu\|$ is the same whether evaluated as
functionals on $\cM_V$ or $\cM_d$. In the latter algebra, $\lambda$ and $\nu$
do lie in the same Gleason part.

\begin{lemma} \label{L:fibers} 
Let $V$ be a variety in $\mb{B}_d$. 
\begin{enumerate}
\item Let $\phi \in \pi^{-1}(\mu)$ for $\mu\in\partial \mb{B}_d$.
Suppose that $\psi \in M(\cM_V)$ satisfies $\|\psi-\phi\|<2$.
Then $\psi$ also belongs to $\pi^{-1}(\mu)$.

\item If $\lambda$ and $\mu$ belong to $V$, then $\|\rho_\mu-\rho_\lambda\| \le 2r <2$,
where $r$ is the pseudohyperbolic distance between $\mu$ and $\lambda$.
\end{enumerate}
\end{lemma}

\begin{proof}
If $\psi\in\pi^{-1}(\nu)$ for $\nu\ne\mu$ in the sphere, then there is an automorphism of
$\mb{B}_d$ that takes $\mu$ to $(1,0,\dots,0)$ and $\nu$ to $(-1,0,\dots,0)$.
Proposition~\ref{prop:auto_implemt_iso} shows that composition by this automorphism is
a completely isometric automorphism. 
So we may suppose that $\mu=(1,0,\dots,0)$ and $\nu=(-1,0,\dots,0)$.
But then 
\[ \|\psi-\phi\| \ge |(\psi-\phi)(Z_1)| = 2 .\]

Similarly, if $\psi = \rho_\lambda$ for some $\lambda \in V$, 
then for any $0 < \ep < 1$, there is an automorphism of
$\mb{B}_d$ that takes $\mu$ to $(1,0,\dots,0)$ and $\nu$ to $(-1+\ep,0,\dots,0)$.
The same conclusion is reached by letting $\ep$ decrease to $0$.

If $\lambda$ and $\mu$ belong to $V$, then there is an automorphism $\gamma$ of $\mb B_d$
sending $\lambda$ to $0$ and $\mu$ to some $v:=(r,0,\dots,0)$ where $0<r<1$ 
is the pseudohyperbolic distance between $\lambda$ and $\mu$.
Given any multiplier $f \in \cM_V$ with $\|f\|=1$,  Proposition~\ref{prop:complete_quotient}
provides a multiplier $\tilde f$ in $\cM_d$ so that $\tilde f|_V=f$ and $\|\tilde f\|=1$.
In particular, $\tilde f\circ\gamma^{-1}$ is holomorphic on $\mb B_d$ and $\|\tilde f\circ\gamma^{-1}\|_\infty \le 1$.
Hence the Schwarz Lemma \cite[Theorem~8.1.4]{RudinBall} shows that 
\[ 
\left| \frac{f(\mu) - f(\lambda)}{1- f(\mu)\ol{f(\lambda)}}\right| = \left| \frac{\tilde f\circ\gamma^{-1}(v) - \tilde f\circ\gamma^{-1}(0)}{1- \tilde{f}\circ\gamma^{-1}(v)\ol{\tilde{f}\circ\gamma^{-1}(0)}}\right| \le  r .\]
Hence
\[ 
 \| \rho_\mu - \rho_\lambda \| = 
 \sup_{\|f\|\le1} |(\rho_\mu-\rho_\lambda)(f)| \le 
 r \sup_{\|f\|\le1} | 1- f(\mu)\ol{f(\lambda)}| 
 \le 2r . \qedhere
\]
\end{proof}

This provides some immediate information about norm continuous maps between
these maximal ideal spaces.

\begin{corollary}\label{C:fibers}
Let $V$ and $W$ be varieties in $\bB_d$, $d<\infty$. 
Suppose that $\phi$ is a continuous algebra homomorphism of $\cM_V$ into $\cM_W$.
\begin{enumerate}
\item Then $\phi^*$ maps each irreducible subvariety of $W$ into $V$ or into a single
fiber of the corona. 

\item If $\phi$ is an isomorphism, and $V$ and $W$ are the disjoint union of finitely 
many irreducible subvarieties, then $\phi^*$ must map $W$ onto $V$.

\item If $\phi$ is an isometric isomorphism, then $\phi^*$ maps $W$ onto $V$ and preserves the
pseudohyperbolic distance.

\end{enumerate}
\end{corollary}

\begin{proof}
(1) Let $W_1$ be an irreducible subvariety of $W$, and let
$\lambda$ be any regular point of $W_1$. We do not assert that $W_1$ is connected.

Suppose that $\phi^*(\rho_\lambda)$ is a point evaluation at some point $\mu$ in $\mb B_d$. 
Then by Proposition~\ref{prop:charM_V}, $\mu$ belongs to $V$.
Since $\phi$ is norm continuous, by Lemma~\ref{L:fibers} it must map the connected 
component of $\lambda$ into a connected component of $V$.

Similarly, suppose that $\rho_\lambda$ is mapped by $\phi^*$ into a fiber of the corona.
Without loss of generality, we may suppose that it is the fiber over $(1,0,\dots,0)$.
Since $\phi$ is norm continuous, by Lemma~\ref{L:fibers} it must map the connected 
component of $\lambda$ into this fiber as well.
Suppose that there is some point $\mu$ in $W_1$ mapped into $V$ or into another fiber.
So the whole connected component of $\mu$ is also mapped into $V$ or another fiber. 
Then the function $h = \phi(Z_1) - 1$ vanishes on the component of $\lambda$ but
does not vanish on the component containing $\mu$.
This contradicts the fact that $W_1$ is irreducible.
Thus the whole subvariety must map entirely into a single fiber or entirely into $V$.

(2) Suppose that $W$ is the union of irreducible subvarieties $W_1, \dots,W_n$.
Fix a point $\lambda \in W_1$. For each $2\le i \le n$, there is a multiplier $h_i \in \cM_{d'}$
which vanishes on $W_i$ but $h_i(\lambda) \ne 0$. 
Hence $h = h_2 h_3 \cdots h_k |_{W}$ belongs to $\cM_W$ and vanishes on 
$\cup_{i=2}^k W_i$ but not on $W_1$. 
Therefore $\phi^{-1}(h) = f$ is a non-zero element of $\cM_V$.
Suppose that $\phi^*(W_1)$ is contained in a fiber over a point in the boundary of
the sphere, say $(1,0,\dots,0)$. 
Since $Z_1-1$ is non-zero on $V$, we see that $(Z_1-1)f$ is not the zero function.
However, $(Z_1-1)f$ vanishes on $\phi^*(W_1)$.
Therefore $\phi((Z_1-1)f)$ vanishes on $W_1$ and on $\cup_{i=2}^k W_i$.
Hence $\phi((Z_1-1)f) = 0$, contradicting injectivity.
We deduce that $W_1$ is mapped into $V$.

By interchanging the roles of $V$ and $W$, we deduce that $\phi^*$ must map $W$ onto $V$.

(3) In the isometric case, we can make use of \cite[Corollary 6.6]{KenYan12}, which says that $\M_V$ and $\M_W$ have strongly unique preduals, from which it follows that $\phi^*$ preserves weak-$*$ continuous functionals. Thus $W$ is mapped into $V$.
Reversing the role of $V$ and $W$ shows that this map is also onto $V$.

The proof of Lemma~\ref{L:fibers}(2) actually yields more information, namely that 
$\|\rho_\lambda-\rho_\mu\|$ is a function of the pseudohyperbolic distance $r$, 
\[ \|\rho_\lambda-\rho_\mu\| = r \sup_{\|f\|\le1} | 1- f(\mu)\ol{f(\lambda)}| .\] 
In the proof of that lemma we only used that the left hand side is less than or equal to the right hand side, but it is easy to see that one obtains equality by choosing a particular $f$. So the fact that the quantities $\|\rho_\lambda-\rho_\mu\|$ and $\sup_{\|f\|\le1} | 1- f(\mu)\ol{f(\lambda)}|$ are preserved by an isometric isomorphism implies that the
pseudohyperbolic distance $r$ is also preserved.
\end{proof}

\begin{remarks} \label{R:fiber}
(1) In a previous version of this paper, we claimed incorrectly that if $\phi$ is a surjective
continuous homomorphism of $\cM_V$ onto $\cM_W$, then $\phi^*$ must map $W$
into $V$. This is false, and we thank Michael Hartz for pointing this out.  
This follows from Hoffman's theory \cite{Hoffman_disks} of analytic disks in 
the corona of $H^\infty$.
There is an analytic map $L$ of the unit disk $\mb D$ into the corona of $M(H^\infty)$,
mapping onto a Gleason part, with the property that $\phi(h)(z) = h (L(z))$ is a homomorphism 
of $H^\infty$ onto itself \cite[ch.X\secsymb 1]{Garnett}.
Therefore the map $\phi^*$ maps the disk into the corona via $L$.

(2) The main obstacle preventing us from establishing part (2) of the corollary in greater
generality is that we do not know that if $\lambda \in W$, then there is an irreducible
subvariety $W_1 \subset W$ containing $\lambda$ and another subvariety $W_2 \subset W$
so that $\lambda \not\in W_2$ and $W=W_1 \cup W_2$. 
As mentioned in the introduction, for any classical analytic variety this is possible
\cite[ch.3, Theorem~1G]{Whitney}. But our definition requires these subvarieties to
be the intersection of zero sets of multipliers. Moreover our proof makes significant use
of these functions. So we cannot just redefine our varieties to have a local definition
as in the classical case even if we impose the restriction that all functions are multipliers.
A better understanding of varieties in our context is needed.

(3) Costea, Sawyer and Wick \cite{CSW} establish a corona theorem for the algebra $\cM_d$ for $d<\infty$.
That is, the closure of the ball $\mb B_d$ in $M(\cM_d)$ is the entire maximal ideal space.
This result may also hold for the quotients $\cM_V$, but we are not aware of any direct proof
deducing this from the result for the whole ball.

A corona theorem for $\cM_V$ would resolve the difficulties in case (2).
The topology on $V = \mb B_d \cap M(\cM_V)$ coincides with the usual one.
In particular, each component has closed complement. The corona theorem
would establish that every open subset of any fiber is in the closure of its complement.
Thus any homeomorphism $\phi^*$ of $M(\cM_W)$ onto $M(\cM_V)$ must take
$W$ onto $V$. 
However it is likely that the corona theorem  for $\cM_V$ is much more difficult than our problem.
\end{remarks}

Now we can deal with the case in which our variety is a finite union of
nice subvarieties, where nice will mean either irreducible or discrete. 

\begin{theorem} \label{theorem:algiso_biholo}
Let $V$ and $W$ be varieties in $\mb{B}_d$, with $d<\infty$,
which are the union of finitely many irreducible varieties and a discrete variety. 
Let $\phi$ be a unital algebra isomorphism of $\cM_V$ onto $\cM_W$. 
Then there exist holomorphic maps $F$ and $G$ from $\mb{B}_{d}$ into $\mb{C}^d$ 
with coefficients in $\cM_d$ such that 
\begin{enumerate}
 \item $F|_{W} = \phi^*|_W \qand  G|_V = (\phi^{-1})^*|_V $
 \item $G \circ F |_W = \id_W \qand F \circ G|_V = \id_V$
 \item $\phi(f) = f \circ F \qfor f \in \cM_V$, and
 \item $\phi^{-1}(g) = g \circ G \qfor g \in \cM_W$. 
\end{enumerate}
\end{theorem}

\begin{proof}
First we show that $\phi^*$ maps $W$ into $V$.
Write 
\[ W = D \cup W_1\cup \dots \cup W_n \]
where $D$ is discrete and each $W_i$ is an irreducible variety.
The points in $D$ are isolated, and thus are mapped into $V$ by 
Lemma~\ref{L:weak_algiso_biholo_discrete}.
A minor modification of Corollary~\ref{C:fibers}(2) deals with the irreducible subvarieties.
Since $D$ is a variety, there is a multiplier $k\in \cM_d$ which vanishes on $D$
and is non-zero at a regular point $\lambda \in W_1$.
Proceed as in the proof of the lemma, but define $f = h_2\dots h_n k$.
Then the argument is completed in the same manner.
Reversing the roles of $V$ and $W$ shows that $\phi^*$ maps $W$ onto $V$.
Similarly, one obtains that $(\phi^{-1})^*$ maps $V$ onto $W$. The remaining statements therefore follow from Proposition \ref{prop:weak_star_homo}. 
%
\end{proof}

\begin{remark}\label{rem:poly_components}
Note that in the above theorem, 
the map $F$ can be chosen to be a polynomial if and only if 
the algebra homomorphism $\phi$ takes the coordinate functions to 
(restrictions of) polynomials; and hence takes polynomials to polynomials. 
Likewise, $F$ can be chosen to have components which are continuous multipliers 
if and only if $\phi$ takes the coordinate functions to continuous multipliers; 
and hence takes all continuous multipliers to continuous multipliers. 
\end{remark}


\begin{corollary}\label{cor:iso_of_Md}
Every algebraic automorphism of $\cM_d$ for $d$ finite is completely isometric, 
and is unitarily implemented.
\end{corollary}

\begin{proof}
The previous theorem shows that every automorphism is implemented as
composition by a biholomorphic map of the ball onto itself,
i.e.\ a conformal automorphism of $\mb B_d$.
Proposition \ref{prop:auto_implemt_iso} shows that these automorphisms
are completely isometric and unitarily implemented.
\end{proof}

Now we consider the isometric case.

\begin{theorem}\label{T:isometric}
Let $V$ and $W$ be varieties in $\mb{B}_d$, with $d<\infty$.
Every isometric isomorphism of $\cM_V$ onto $\cM_W$ is completely isometric, 
and thus is unitarily implemented.
\end{theorem}

\begin{proof}
Let $\phi$ be an isometric isomorphism of $\cM_V$ onto $\cM_W$.
By Corollary~\ref{C:fibers}(3), $\phi^*$ maps $W$ onto $V$ and preserves the pseudohyperbolic distance.
Let $F$ be the function constructed as in Theorem~\ref{theorem:algiso_biholo}.
As in Lemma~\ref{L:cc_homo_holo} and Theorem~\ref{theorem:algiso_biholo},
$F$ is a biholomorphism of $W$ onto $V$ and $\phi(h) = h \circ F$.

After modifying both $V$ and $W$ by a conformal automorphism of the ball,
we may assume that $0$ belongs to both $V$ and $W$, and that $F(0)=0$.
Set $w_0=0$ and choose a basis $w_1,\dots,w_k$ for $\spn W$. 
Let $v_p=F(w_p)$ for $1 \le p \le k$.

Suppose that $\|w_p\|=r_p$. This is the pseudohyperbolic distance to $w_0=0=v_0$,
so $\|v_p\|=r_p$ as well. Write $v_p/r_p = \sum_{j=1}^d c_j e_j$.
Let $h_p(z) = \ip{z,v_p/r_p} = \sum_{j=1}^d \ol{c}_j Z_j(z)$. 
This is a linear function on $V$, and thus lies in $\cM_V$.
Since $Z$ is a row contraction, $f$ has norm at most one.
Therefore $k_p:=\phi(h_p) = h_p \circ F$ has norm at most one in $\cM_W$.

Now let $w_{k+1}=w$ be an arbitrary point in $W$, and set $v_{k+1}=v=F(w)\in V$.
By a standard necessary condition for interpolation \cite[Theorem~5.2]{AM02}, the fact that $\|k_p\| \le1$ 
means that in particular interpolating at the  points $w_0,\dots,w_k,w_{k+1}$, we obtain
\[
 0 \le \begin{bmatrix} \frac{1-h_p(v_i)\ol{h_p(v_j)}}{1-\ip{w_i,w_j}} \end{bmatrix}_{0 \le i,j \le k+1} .
\]
In particular, look at the $3\times 3$ minor using rows $0,p,k+1$ to obtain
\[
 0 \le 
 \begin{bmatrix} \ 
 1&1&1\\[.6ex]\ 
 1&1& \frac{1-\ol{\ip{v,v_p}}}{1-\ip{w_p,w}}\\[1.6ex] \ 
 1&\frac{1-\ip{v,v_p}}{1-\ip{w,w_p}}&\frac{1-|\ip{v,v_p/r_p})|^2}{1-\|w\|^2}
 \end{bmatrix}
\]
By the Cholesky algorithm, we find that $\frac{1-\ip{v,v_p}}{1-\ip{w,w_p}} = 1$.
Therefore
\[ \ip{v,v_p} = \ip{w,w_p} \qfor 1 \le p \le k .\]

In particular, we obtain
\[ \ip{v_i,v_j}=\ip{w_i,w_j} \qfor  1 \le i,j \le k .\]
Therefore there is a unitary operator $U$ acting on $\mb C^d$ 
such that $Uw_i=v_i$ for $1 \le i \le k$.
Now since $w\in W$ lies in $\spn\{w_1,\dots,w_k\}$, it is uniquely determined by
the inner products $\ip{w,w_i}$ for $1 \le i \le k$. 
Since $v$ has the same inner products with $v_1,\dots,v_k$, we find that
$Uw = P_N v$ where $N = \spn\{v_1,\dots,v_k\}$. 
However we also have 
\[ \|v\|=\|w\|=\|Uw\|=\|P_Nv\| ; \]
whence $v=Uw$.

Therefore $F$ agrees with the unitary $U$, and hence $\phi$ is implemented by
an automorphism of the ball.
So by Proposition~\ref{prop:auto_implemt_iso}, $\phi$ is completely isometric and 
is unitarily implemented.
\end{proof}

With these results in hand, we may repeat the arguments in 
\cite[Section 11.3]{DRS} word for word to obtain the following automatic continuity result.
Recall that the weak-operator and the weak-$*$ topologies on $\cM_V$ 
coincide by Lemma~\ref{lem:weak_weakstar}.

\begin{theorem}\label{theorem:weak_continuity}
Let $\phi : \cM_V \rightarrow \cM_W$, for $d<\infty$, be a unital algebra isomorphism given
by composition: $\phi(h) = h \circ F$ where $F$ is a holomorphic map of $W$ onto $V$
whose coefficients are multipliers.
Then $\phi$ is continuous with respect to the weak-operator and the weak-$*$ topologies.
\end{theorem}

\section{Examples}

In this section, we examine a possible converse to 
Theorem~\ref{theorem:algiso_biholo} in the context of a number of examples.
What we find is that the desired converse is not always true.
That is, suppose that $V$ and $W$ are varieties in $\mb B_d$
and $F$ and $G$ are holomorphic functions on the ball satisfying the
conclusions of Theorem~\ref{theorem:algiso_biholo}.
We are interested in when this implies that the algebras $\cM_V$ and $\cM_W$
are isomorphic.

\subsection*{Finitely many points in the ball}
Let $V = \{v_1, \ldots, v_n\} \subseteq \mb{B}_d$. 
Then $\cA_V = \cM_V$ and they are both isomorphic to $\ell^\infty_n = C(V)$. 
The characters are evaluations at points of $V$. 
If $W$ is another $n$ point set in $\mb{B}_{d'}$, then 
$\cM_W$ is isomorphic to $\cM_V$. 
Also, there are (polynomial) maps $f: \mb{B}_d \rightarrow \mb{C}^{d'}$ 
and $g: \mb{B}_{d'} \rightarrow \mb{C}^{d}$ which are inverses 
of one another when restricted to $V$ and $W$. 
And if $W$ is an $m$ point set, $m\neq n$, then obviously 
$\cM_V$ is not isomorphic to $\cM_W$, 
and there also exists no biholomorphism. 
In this simple case we see that $\cM_V \cong \cM_W$ if and only if 
there exists a biholomorphism, and this happens if and only if $|W|=|V|$.

Nevertheless, the situation for finite sets is not ideal.
Let $V$ and $W$ be finite subsets of the ball, and let $F : W \rightarrow V$ be a biholomorphism. It is natural to hope that the norm of the induced isomorphism can be bounded in terms of the multiplier norm of $F$. The following example shows that this is not possible.

\begin{example}
Fix $n \in \mb{N}$ and $r \in (0,1)$. 
Put $\xi = \exp(\frac{2\pi i}{n})$ and let 
\[ V = \{0\} \cup \{r \xi^j \}_{j=1}^n , \]
and 
\[ W = \{0\} \cup \{\frac{r}{2} \xi^j \}_{j=1}^n . \]
The map $F(z) = 2z$ is a biholomorphism of $W$ onto $V$ 
that extends to an $H^\infty$ function of multiplier norm $2$. 
We will show that the norm of the induced isomorphism 
$\cM_V \rightarrow \cM_W$, given by $f \mapsto f \circ F$, is at least $2^n$. 

Consider the following function in $\cM_V$:
\[ f(0)=0 \qand f(r \xi^j) = r^n \qfor 1 \le j \le n .\]
We claim that the multiplier norm of $f$ is $1$. 
By Proposition \ref{prop:complete_quotient}, $\|f\|$ is the minimal norm
of an $H^\infty$ function that interpolates $f$. 
The function $g(z) = z^n$ certainly interpolates and has norm 1.
We will show that it is of minimal norm. 

The Pick matrix associated to the problem of interpolating $f$ on $V$
by an $H^\infty$ function of norm 1 is
\[
\begin{bmatrix}
1 & 1 & 1 & \cdots & 1 \\
1 & \frac{1 - r^{2n}}{1 - r^2 \xi \ol{\xi}} & \frac{1 - r^{2n}}{1 - r^2 \xi \ol{\xi^2}} & \cdots & \frac{1 - r^{2n}}{1 - r^2 \xi \ol{\xi^n}} \\
1 & \frac{1 - r^{2n}}{1 - r^2 \xi^2 \ol{\xi}} & \frac{1 - r^{2n}}{1 - r^2 \xi^2 \ol{\xi^2}} & \cdots & \frac{1 - r^{2n}}{1 - r^2 \xi^2 \ol{\xi^n}} \\
\vdots & \vdots & \vdots & \ddots & \vdots \\
1 & \frac{1 - r^{2n}}{1 - r^2 \xi^n \ol{\xi}} & \frac{1 - r^{2n}}{1 - r^2 \xi^n \ol{\xi^2}} & \cdots & \frac{1 - r^{2n}}{1 - r^2 \xi^n \ol{\xi^n}} 
\end{bmatrix} .
\]
To show that $g$ is the (unique) function of minimal norm that interpolates $f$, 
it suffices to show that this matrix is singular. 
(We are using well known facts about Pick interpolation. 
See Chapter 6 in \cite{AM02}). 

We will show that the lower right principal sub-matrix
\[
 A = \Bigg[ \frac{1 - r^{2n}}{1 - r^2 \xi^i \ol{\xi^j}} \Bigg]_{i,j=1}^n 
\]
has the vector $(1,\ldots, 1)^t$ as an eigenvector with eigenvalue $n$.
It follows that $(n,-1,-1, \ldots, -1)^t$ is in the kernel of the Pick matrix.
(The matrix $A$ is invertible, so the Pick matrix has rank $n$). 

Indeed, for any $i$,
\begin{align*}
\sum_{j=1}^n \frac{1 - r^{2n}}{1 - r^2 \xi^i \ol{\xi^j}} 
&= (1-r^{2n})\sum_{j=1}^n \sum_{k=0}^\infty (r^2 \xi^i \ol{\xi^j})^k \\
&= (1-r^{2n}) \sum_{k=0}^\infty \sum_{j=1}^n r^{2k} \xi^{ik} \ol{\xi^{jk}} \\
&= (1-r^{2n}) \sum_{m=0}^\infty n r^{2mn} \xi^{imn}  \\
&= n \frac{1-r^{2n}}{1-r^{2n}} = n.
\end{align*}
We used the familiar fact that $\sum_{j=1}^n \xi^{jk}$ is equal 
to $n$ for $k\equiv 0 \pmod{n}$ and equal to $0$ otherwise. 
Therefore $\|f\| = 1$.

Now we will show that $f \circ F \in \cM_W$ has norm $2^n$, where $F(z) = 2z$. 
The function $f \circ F$ is given by 
\[ f \circ F(0) = 0 \qand  f \circ F(\tfrac{r}{2} \xi^j) = r^n \qfor 1 \le j \le n .\] 
The unique $H^\infty$ function of minimal norm that interpolates $f \circ F$ is 
$h(z) = 2^n z^n$.  This follows from precisely the same reasoning as above. 
Therefore the isomorphism has norm at least $2^n$. 
\end{example}

\subsection*{Blaschke sequences}
We will now provide an example of two discrete varieties which are
biholomorphic but yield non-isomorphic algebras.

\begin{example}\label{ex:interp_non_interp_hol}
Let 
\[ v_n = 1-1/n^2 \qand  w_n = 1-e^{-n^2} \qfor n \ge 1 .\]
Set $V = \{v_n\}_{n=1}^\infty$ and $W = \{w_n\}_{n=1}^\infty$. 
Both $V$ and $W$ satisfy the Blaschke condition so they are 
analytic varieties in $\mb{D}$.  
Let $B(z)$ be the  Blaschke product with simple zeros at points in $W$. 
Define
\[
 h(z) = 1 - e^{\frac{1}{z-1}},
\]
and 
\[
 g(z) = \frac{\log(1-z)+1}{\log(1-z)} \Big(1 - \frac{B(z)}{B(0)} \Big).
\]
Then $g,h \in H^\infty$ and they satisfy 
\[ h \circ g|_W = \id_W \qand g \circ h |_V = \id_V .\]
However, by the corollary in \cite[p.204]{Hoffman},  $W$ is an interpolating 
sequence and $V$ is not.
Thus the algebras $\cM_V$ and $\cM_W$ cannot be similar by a map 
sending normalized kernel functions to normalized kernel functions.
The reason is that the normalized kernel functions corresponding to an interpolating 
sequence form a Riesz system, while those corresponding to a non-interpolating 
sequence do not. 
In fact, $\cM_V$ and $\cM_W$ cannot be isomorphic via {\em any} isomorphism, 
as we see below. 
\end{example}

\begin{theorem}\label{theorem:iso_interpolating}
Let $V = \{v_n\}_{n=1}^\infty \subseteq \mb{B}_d$, with $d<\infty$, be a sequence 
satisfying the Blaschke condition $\sum (1-\|v_n\|) < \infty$. 
Then $\cM_V$ is isomorphic to $\ell^\infty$ if and only if $V$ is interpolating.
\end{theorem}

\begin{proof}
By definition, $V$ is interpolating if and only if $\cM_V$ is isomorphic 
to $\ell^\infty$ via the restriction map. 
It remains to prove that if $V$ is not an interpolating sequence, 
then $\cM_V$ cannot be isomorphic to $\ell^\infty$ via any other isomorphism. 

Let $V$ be a non-interpolating sequence, and let $W$ be any interpolating sequence. 
If $\cM_V$ is isomorphic to $\ell^\infty$, then it is isomorphic to $\cM_W$. 
But by Lemma \ref{L:weak_algiso_biholo_discrete}, this isomorphism must 
be implemented by composition with a holomorphic map, showing that $\cM_V$ is isomorphic to $\ell^\infty$ via the restriction map. This is a contradiction.
\end{proof}

\begin{remark}
We require the Blaschke condition to insure that $V$ is a variety of the type 
we consider, i.e., a zero set of an ideal of multipliers 
(see \cite[Theorem 1.11]{AriasLat10a}). 
Any discrete variety in $\mb{D}$ satisfies this condition.  
\end{remark}

\subsection*{Curves}

Let $V$ be a variety in $\mb{B}_d$. 
If $\cM_V$ is  isomorphic to $H^\infty(\mb{D})$, then by 
Theorem~\ref{theorem:algiso_biholo} we know that 
$V$ must be biholomorphic to the disc. 
To study the converse implication, we shall start with a disc biholomorphically 
embedded in a ball and try to establish a relationship between the associated 
algebras $\cM_V$ and its reproducing kernel Hilbert space $\cF_V$ 
and $H^\infty(\mb{D})$ and $H^2(\mb{D})$. 

Suppose that $h$ is a holomorphic map from the disc $\mb D$ into $\mb B_d$ 
such that $h(\mb{D}) = V$, and that there exists a holomorphic map 
$g: \mb{B}_d \to \mb{C}$ such that $g|_V = h^{-1}$. 

The following result shows that in many cases, the desired isomorphism exists
\cite{APV03}. See \cite[\secsymb 2.3.6]{ARS} for a strengthening to planar domains, and
a technical correction.

\begin{theorem}[Alpay-Putinar-Vinnikov] \label{T:APV}
Suppose that $h$ is an injective holomorphic function of $\mb D$ onto 
$V \subset \mb B_d$ such that
\begin{enumerate}
\item $h$ extends to a $C^1$ function on $\ol{\mb D}$,
\item $\|h(z)\| = 1$ if and only if $|z|=1$,
\item $\ip{h(z),h'(z)} \ne 0$ when $|z|=1$.
\end{enumerate}
Then $\cM_V$ is isomorphic to $H^\infty$.
\end{theorem}

Condition (3) should be seen as saying that $V$ meets the 
boundary of the ball non-tangentially. We do not know whether
such a condition is necessary.

The authors of \cite{APV03} were concerned with extending multipliers
on $V$ to multipliers on the ball. This extension follows from
Proposition~\ref{prop:complete_quotient}.

By the results of  Section \ref{S:cc_iso}, there is no loss of generality 
in assuming that $h(0) = 0$, and we do so.
Define a kernel $\tilde k$ on $\mb{D}$ by
\[ \tilde k(z,w) = k(h(z),h(w)) = \frac{1}{1 - \lel h(z), h(w)\rir} . \]
Let $\H$ be the RKHS determined by $\tilde k$. 
Write $\tilde k_w$ for the function $\tilde k(\cdot,w)$. 

The following routine lemma shows that we can consider this new kernel
on the disc instead of $\cF_V$.

\begin{lemma}\label{lem:k_tilde}
The map $\tilde k_z \mapsto k_{h(z)}$ extends to a unitary map $U$ of $\H$ onto $\cF_V$. 
Hence, the multiplier algebra $\Mult(\H)$ is unitarily equivalent to $\cM_V$. 
This equivalence is implemented by composition with $h$:
\[ U^* M_f U = M_{f \circ h} \qfor f \in \cM_V .\]
\end{lemma}

\begin{proof}
A simple computation shows that 
\begin{align*}
\big\| \sum_i c_i \tilde k_{z_i} \big\|^2 
&= \sum_{i,j} \frac{c_i \ol{c_j}}{1 - \lel h(z_i), h(z_j) \rir }  
 = \big\| \sum_i c_i k_{h(z_i)} \big\|^2 .
\end{align*}
So we get a unitary $U : \H \rightarrow \cF_V$. 
As in the proof of Proposition \ref{prop:auto_implemt_iso}, 
for all $f \in \cM_V$ we have $U^* M_f U = M_{f \circ h}$.
\end{proof}

Our goal in this section is to study conditions on $h$ 
which yield a natural isomorphism of the RKHSs $\H$ and $H^2(\mb{D})$. 
The first result is that the Szego kernel $k_z$ dominates the kernel $\tilde k_z$.

\begin{lem} \label{L:NP}
Suppose that $h$ is a holomorphic map of $\bD$ into $\bB_d$.
Then for any finite subset $\{ z_1,\dots,z_n \} \subset \bD$, 
\[  
 \bigg[ \frac 1 {1-\ip{h(z_j),  h(z_i)} } \bigg] \le \bigg[ \frac 1 {1-z_j \ol{z_i}} \bigg] .
\]
\end{lem}

\begin{proof}
Observe that $h(z)/z$ maps $\bD$ into $\ol{\bB_d}$ by Schwarz's Lemma
\cite[Theorem 8.1.2]{RudinBall}. Thus by the matrix version of the
Nevanlinna-Pick Theorem for the unit disk, we obtain that
\begin{align*}
 0 &\le  \left[ \frac {1-\ip{h(z_j)/z_j,  h(z_i)/z_i}}  {1 - z_j \ol{z_i}} \right] 
 = \left[ \frac1{ z_j \ol{z_i}} \right] \circ 
 \left[ \frac {1-\ip{h(z_j),  h(z_i)}} {1 - z_j \ol{z_i}} - 1  \right] .
\end{align*}
Here $\circ$ represents the Schur product. But $\left[ \frac1{ z_j \ol{z_i}} \right]$ and
its Schur inverse $\big[ z_j \ol{z_i} \big]$ are positive.  Therefore the second matrix
on the right is positive. This can be rewritten as 
\[ \Big[ \ \ 1 \ \ \Big] \le \left[ \frac {1-\ip{h(z_j),  h(z_i)}} {1 - z_j \ol{z_i}}  \right]  \]
where $\big[\, 1\, \big]$ represents an $n\times n$ matrix of all 1's.
Now 
\[ 
 \left[ \frac 1{1-\ip{h(z_j),  h(z_i)}} \right] = \left[ \ip{\tilde k_{z_i}, \tilde k_{z_j}} \right] \ge 0 .
\]
So the Schur multiplication by this operator to the previous inequality yields
\[
  \left[ \frac 1{1-\ip{h(z_j),  h(z_i)}} \right] 
  \le \left[ \frac 1 {1 - z_j \ol{z_i}}  \right] . \qedhere
\]
\end{proof}

We obtain the well-known consequence that there is a contractive map  
of $H^2$ into $\cH$. 

\begin{prop} \label{P:bounded}
The linear map $R$, defined by $Rk_z = \tilde k_z$ for $z \in \mb D$, 
from $\spn\{ k_z : z \in \bD \}$ to $\spn\{ \tilde k_z : z \in \bD \}$
extends to a contractive map from $H^2$ into $\H$. 
\end{prop}

\begin{proof}
This follows from an application of Lemma~\ref{L:NP}. 
Given $a_i \in \bC$, let $\ba = (a_1,\dots,a_n)^t$.  
Observe that
\begin{align*}
 \| R \sum_{i=1}^n a_i k_{z_i} \|^2 &=
 \| \sum_{i=1}^n a_i \tilde k_{z_i} \|^2 =
 \sum_{i,j=1}^n a_i \ol{a_j} \ip{ \tilde k_{z_i}, \tilde k_{z_j} } \\&
 = \bip{ \Big[ \ip{ \tilde k_{z_i}, \tilde k_{z_j} } \Big] \ba, \ba } 
 \le \bip{ \Big[ \ip{ k_{z_i}, k_{z_j} } \Big] \ba, \ba } \\&
 = \sum_{i,j=1}^n a_i \ol{a_j} \ip{ k_{z_i}, k_{z_j} }
 = \| \sum_{i=1}^n a_i k_{z_i} \|^2
\end{align*}
Hence $R$ is contractive, and extends to $H^2$ by continuity.
\end{proof}

\begin{eg}\label{eg:ball}
Let $h : \mb{D} \to \mb{B}_d$ be given by 
\[
h(z) = (a_1 z, a_2 z^{n_2}, \ldots, a_d z^{n_d}) ,
\]
where $a_1 \ne 0$ and $\sum_{l=1}^d |a_l|^2 = 1$. Let $V = h(\mb{D})$. 
Then $\cM_V$ is similar to $H^\infty(\mb{D})$, and $\cM_V = H^\infty(V)$. 
Moreover, $\A_V$ is similar to $\AD$.
This follows from Theorem~\ref{T:APV}, but we will provide a direct argument.

First observe that for $p\ge N = \max\{n_l: 1 \le l \le d\}$, we have
\begin{align*}
 \frac{\ip{h(z),h(w)} - z^p\bar{w}^p}{1-z\bar{w}} &=
 \sum_{l=1}^d |a_l|^2 \bigg( \frac{z^{n_l}\bar{w}^{n_l} - z^p\bar{w}^p}{1-z\bar{w}} \bigg)\\
 &= \sum_{l=1}^d |a_l|^2 z^{n_l}\bar{w}^{n_l}
  \bigg( \frac{1-z^{p-n_l}\bar{w}^{p-n_l}}{1-z\bar{w}} \bigg)
\end{align*}
Therefore if $z_1,\dots,z_k$ are distinct points in $\mb D$, the $k\times k$ matrix
\[
 A_p :=
 \bigg[ \frac{\ip{h(z_i),h(z_j)} - z_i^p\bar{z}_j^p}{1-z_i\bar{z}_j} \bigg] =
 \sum_{l=1}^d |a_l|^2  \bigg[ z_i^{n_l} \bar{z}_j^{n_l} \bigg] \circ 
 \bigg[ \frac{1-z_i^{p-n_l}\bar{z}_j^{p-n_l}}{1-z\bar{w}} \bigg]
\]
is positive definite because the second matrix on the right is positive by
Pick's condition, and the Schur product of positive matrices is positive.

Since the first coordinate of $h$ is injective, we see that $h$ is injective.
Moreover, 
\[ \|h^{-1}\|_{\cM_V} \le \|a_1^{-1} z_1\|_{\cM} = |a_1|^{-1} =: C. \]
Since the kernel for $\cF_V$ is a complete NP kernel, applying this
to $(h^{-1})^{2^{n-1}}$ yields the positivity of the matrices
\[  \bigg[\frac{C^{2^n}- z_i^{2^{n-1}}\bar{z}_j^{2^{n-1}}}{1- \ip{h(z_i),h(z_j)}} \bigg] .\]
Since $z^{2^{n-1}}$ has norm one, the Pick condition shows that
\[  \bigg[\frac{1- z_i^{2^{n-1}} \bar{z}_j^{2^{n-1}}}{1- z_i \bar{z}_j} \bigg] \ge 0.\]
Thus we obtain positive matrices 
\begin{align*}
 H_n &:=
 \bigg[ \frac{C^{2^n} - z_i^{2^{n-1}}  \bar{z}_j^{2^{n-1}}}{1- \ip{h(z_i),h(z_j)}} \bigg]
 \circ \bigg[ \frac{1 - z_i^{2^{n-1}}  \bar{z}_j^{2^{n-1}}}{1- z_i \bar{z}_j}  \bigg] \\&=
 \bigg[ \frac{C^{2^n} - (C^{2^n} +1) z_i^{2^{n-1}}  \bar{z}_j^{2^{n-1}}
 + z_i^{2^n} \bar{z}_j^{2^n}} {(1 - \ip{h(z_i),h(z_j)})(1 - z_i \bar{z}_j)}  \bigg]
\end{align*}

Choose $M$ so that $2^M \ge N$. 
We form a telescoping sum of positive multiples of the $H_n$'s:
\[
 0 \le \sum_{n=1}^M b_n H_n = 
 \bigg[ \frac{(D-1) -D z_i \bar{z}_j + z_i^{2^M} \bar{z}_j^{2^M}} 
 {(1 - \ip{h(z_i),h(z_j)})(1 - z_i \bar{z}_j)}  \bigg] =: H
\]
where $b_M=1$, $b_n = \prod_{k=n+1}^M (C^{2^k}+1)$ for $1 \le n < M$ and 
$D = \prod_{k=1}^M (C^{2^k}+1)$.
Thus
\begin{align*}
  \bigg[\frac{D}{1\!-\! \ip{h(z_i),h(z_j)}} \bigg] - \bigg[ \frac{1}{1 \!-\! z_i \bar{z}_j} \bigg]
  &= \bigg[ \frac{(D\!-\!1) \!-\! D z_i \bar{z}_j \!+\! \ip{h(z_i),h(z_j)}}
 {(1 \!-\! \ip{h(z_i),h(z_j)})(1 \!-\! z_i \bar{z}_j)}  \bigg] \\[1ex]
 &= H + A_{2^M}\circ\bigg[\frac1{1 \!-\! \ip{h(z_i),h(z_j)}}\bigg] \ge 0.
\end{align*}

This inequality shows that the two kernels $k_z$ and $\tilde k_z$ are comparable.
The argument of Proposition~\ref{P:bounded} shows that $\|R^{-1}\| \le D$.
In particular, $R$ yields an isomorphism of the two RKHSs $H^2$ and $\cH$.
This yields the desired isomorphism of $H^\infty$ and $\cM_V$.

This isomorphism is not isometric.
Indeed, if it were, then we would have $\|h^{-1}\|_{\cM_V} = \|z\|_\infty = 1$.
This would imply that 
\[  0 \le \bigg[\frac{1- z_i\bar{z}_j}{1- \ip{h(z_i),h(z_j)}} \bigg] .\]
Thus arguing as in Lemma~\ref{L:NP}, we obtain 
\[
 \bigg[ \frac 1 {1-z_j \ol{z_i}} \bigg]  \le \bigg[ \frac 1 {1-\ip{h(z_j),  h(z_i)} } \bigg] .
\]
But then the map $R$ would be unitary, and the algebras would be completely isometric.
So by Lemma \ref{L:comp_iso_iso_biholo}, the map $h$ would map onto an
affine disk---which it does not do. \qed
\end{eg}

\subsection*{A class of examples in $\mb{B}_\infty$}

We will now exhibit biholomorphisms of $\mb{D}$ into $\mb{B}_\infty$, some of which
yield an isomorphism and some which do not. 

Let $\{b_n\}_{n=1}^\infty$ be a sequence of complex numbers with $\sum |b_n|^2 = 1$ 
and $b_1 \neq 0$.
Let  $h : \mb{D} \rightarrow \mb{B}_\infty$ be given by 
\[
 h(z) = (b_1 z, b_2 z^2, b_3 z^3, \ldots)  .
\]
Note that $h$ is analytic (because it is given by a power series in the disc), 
with the analytic inverse:
\[
  g(z_1, z_2, z_3,  \ldots ) = z_1/b_1 . 
\] 
The set $V = h(\mb D)$ is the variety in $\mb{B}_\infty$ determined by the equations 
\[ z_k = \frac{b_k}{b_1} z_1^k \qfor k \ge 2. \]

As above let
\[
 \tilde k(z,w) = \frac{1}{1 - \langle h(z), h(w) \rangle} , 
\]
and let $\H$ be the RKHS determined by $\tilde k$. 
By Lemma \ref{lem:k_tilde}, $\H$ is equivalent to $\cF_V$. 
The special form of $h$ allows us to write 
\[
 \frac{1}{1 - \ip{h(z), h(w)}} = 
 \sum_{n=0}^\infty \Big(\sum_{i=1}^\infty |b_i|^2 z^i \ol{w}^i\Big)^n = 
 \sum_{n=0}^\infty a_n (z\ol{w})^n .
\]
By a basic result in RKHSs, $\tilde k(z,w) = \sum e_n(z) \ol{e_n(w)}$ 
where $\{e_n\}$ is an orthonormal basis for $\tilde k$ 
(see Proposition 2.18 of \cite{AM02}).
Hence $\H$ is the space of holomorphic functions on $\mb{D}$ with 
orthonormal basis $\{\sqrt{a_n} z^n\}_{n=0}^\infty$. 

The map $R$ defined in Proposition~\ref{P:bounded} is a contraction.
Observe that $R^*: \cH \rightarrow H^2$ is given 
by composition with the identity mapping because
\[ 
 (R^*f)(z) = \ip{ R^* f, k_z } = \ip{ f, R k_z } = \ip{ f, \tilde k_z } = f(z). 
\]
It is easy to see that the issue is whether $R$ is bounded below.
Since $\|z^n\|_{H^2}=1$ and $\|z^n\|_{\cH} = 1/\sqrt{a_n}$, we get:

\begin{proposition}\label{prop:HequivH2}
$\H$ is equivalent to $H^2$ via $R$ if and only if there are constants 
$0 < c < C$ so that $c \le a_n \le C$  for $n\ge0$. 
\end{proposition}

The coefficients $a_n$ are determined by the sequence $\{|b_n|\}_{n=1}^\infty$, 
and can be found recursively by the formulae 
\begin{equation}\label{eq:a_n}
 a_0 = 1 \qand a_n = |b_1|^2 a_{n-1} + \ldots + |b_n|^2 a_0 \qfor n\ge1.
\end{equation} 
The logic behind this recursion is that the term $a_n(z\ol{w})^n$ gets 
contributions from the sum 
\[
 \sum_{k=1}^n\Big(\sum_{i=1}^n |b_i|^2 z^i \ol{w}^i\Big)^k = 
 \Big(\sum_{i=1}^n |b_i|^2 z^i \ol{w}^i\Big) \sum_{k=1}^n 
 \Big(\sum_{i=1}^n |b_i|^2 z^i \ol{w}^i\Big)^{k-1} .
\]
Every $|b_i|^2 z^i \ol{w}^i$ from the factor $\sum_{i=1}^n |b_i|^2 z^i \ol{w}^i$ 
needs to be matched with the $(z \ol{w})^{n-i}$ term from the factor 
$\sum_{k=1}^n (\sum_{i=1}^n |b_i|^2 z^i \ol{w}^i)^{k-1}$, 
which has coefficient precisely $a_{n-i}$.
It follows by induction from equation (\ref{eq:a_n}) that $a_n \le 1$.
This provides an alternative proof of Proposition \ref{P:bounded} in this special case. 

We will now construct a sequence $\{b_n\}_{n=1}^\infty$ that makes 
$\liminf a_n > 0$, and another sequence 
that makes $\liminf a_n = 0$. By Proposition \ref{prop:HequivH2}, this will show that there are choices of $\{b_n\}_{n=1}^\infty$ for which $\H$ and $H^2$ are naturally isomorphic, and there are choices for which they are not. 

\begin{eg}
Define $b_n = (1/2)^{n/2}$ for $n\ge1$. 
It follows from the recursion relation (\ref{eq:a_n}) that $a_n = 1/2$ for $n > 1$. 
Thus $R^*$ is bounded below, showing that $\H$ and $H^2$ are naturally isomorphic. 
\end{eg} 

\begin{eg}
We will choose a rapidly increasing sequence $\{n_k\}_{k=1}^\infty$ with $n_1=1$ and define the sequence $\{b_n\}_{n=1}^\infty$ by 
\[
b_{m} = \begin{cases}
(1/2)^{k/2}  & \textrm{ if } m = n_k \\
0 & \textrm{ otherwise }.
\end{cases}
\]
The sequence $\{n_k\}_{k=1}^\infty$ will be defined recursively so that 
$a_{n_{k}-1} \leq 1/k$.

We begin with $n_1 = 1$ and $a_0 = 1$. 
Suppose that we have already chosen $n_1, \ldots, n_k$. 
This means that we have already determined the sequence $b_1, \ldots, b_{n_k}$, 
but the tail $b_{n_k+1}, b_{n_k+2}, \ldots$ is yet to be determined. 
We compute
\[
\sum_{m=1}^{n_k} b_m^2 = \sum_{j=1}^k b_{n_j}^2 = \sum_{j=1}^{k} 1/2^j = r < 1.
\]
Thus, if $b_{n_k+1} = b_{n_k+2} = \ldots = b_{(N+1)n_k} = 0$, 
then it follows from (\ref{eq:a_n}) that $a_{(N+1)n_k} \leq r^{N}$ 
(recall that $a_n \leq1$ for all $n$). 
Therefore we may choose $N$ so large that $a_{(N+1)n_k} \leq (k+1)^{-1}$, 
and we set $n_{k+1} = (N+1)n_k+1$. 

Our construction yields a sequence $\{b_n\}_{n=1}^\infty$ so that $\liminf a_n = 0$. 
Thus the kernel for the analytic disk $V$ so defined is not similar to $H^2$.

We do not know whether $\cM_V$ is isomorphic to $H^\infty$ or not. 
We suspect that it isn't.
\end{eg} 

\begin{remark}
Suppose that there is some $N$ such that $b_n = 0$ for all $n>N$. 
Then the mapping $h : \mb{D} \rightarrow \mb{B}_\infty$ given by 
\[ h(z) = (b_1 z, b_2 z^2, b_3 z^3, \ldots) \]
can be considered as a mapping into $\mb{B}_N$. 
Equation (\ref{eq:a_n}) implies that for $n > N$, $a_n$ will always remain 
between the minimum and the maximum of $a_0, a_1, \ldots, a_N$. 
Therefore, the conditions of Proposition \ref{prop:HequivH2} are fulfilled, 
and $\H$ is equivalent to $H^2$ via $R$.
This is an alternate argument to obtain Example \ref{eg:ball}. 
\end{remark}

\section{The norm closed algebras $\cA_V$}

Let $\cA_V$ be the norm closure of the polynomials in $\cM_V$. 
Define a closed ideal $I_V$ of $\cA_V$ by  
\[ 
 I_V = \{f \in \cA : f(\lambda) = 0 \FORAL \lambda \in V \} .
\] 
It is natural to study the algebra $\cA_V$ and  
the quotient algebra $\cA / I_V$; 
and to ask whether these algebras can be identified. 
We make a few remarks about the subtleties involved.
These subtleties are the reason why in general it seems that 
the algebras $\cM_V$ will be more amenable to general study.

Consider the following \textit{special assumption}, which we will usually 
assume when considering $\cA_V$:
\be\label{eq:special_assumption}
 [I_V H^2_d] = [J_V H^2_d]. 
\ee

By Theorem \ref{thm:DavPitts2}, this is equivalent to assuming 
that the \wot-closure of $I_V$ is $J_V$. 
In function theoretic terms, this means that every $f \in J_V$ is 
the bounded pointwise limit of a net of functions in $I_V$. 
It is not clear when this happens in general. 
A large class of varieties for which this condition holds is the class of 
homogeneous varieties \cite[Section 6]{DRS}. 
Another class is described below.

In dimension $d=1$, the analytic varieties are sequences of points 
which are either finite or satisfy the Blaschke condition. 
For such a sequence $V$, let us denote $S(V) = \ol{V} \cap \mb{T}$. 
When the Lebesgue measure of $S(V)$ is positive, 
there is no nonzero $f \in A(\mb{D})$ that vanishes on $V$ 
because a non-zero function in the disk algebra must be non-zero a.e.\ on
the unit circle. So $I_V = 0$. 
On the other hand, $J_V \neq 0$, because it contains the Blaschke 
product of the sequence $V$.
So it cannot be the \wot-closure of $I_V$. 
In particular, the special assumption (\ref{eq:special_assumption}) is not always satisfied. 
If the Lebesgue measure $|S(V)|$ of $S(V)$ is zero, 
then the special assumption is valid.

\begin{lemma}\label{lem:indisc}
Let $V$ be an analytic variety in $\mb{D}$ such that $S(V)$ has zero measure. 
Then the ideal $J_V$ is the \wot-closure of $I_V$.
\end{lemma}

\begin{proof}
Let $B$ be the Blaschke product with simple zeros on $V$.
It suffices to construct for every $f \in J_V = BH^\infty$, 
a bounded sequence in $I_V$ converging pointwise to $f$. 
Factor $f = Bh$ with $h\in H^\infty$.
By a theorem of Fatou there is an analytic function $g$ with 
$\re g \ge 0$ such that $e^{-g}$ is in $A(\mb{D})$ and 
vanishes precisely on $S(V)$. 
Define 
\[ f_n(z) = B(z)\, e^{-g(z)/n} h((1-\tfrac1n)z) \qfor n\ge1. \]
This sequence belongs to $\AD$, is bounded by $\|f\|_\infty$, 
and converges to $f$ uniformly on compact subsets of the disk.
Hence it converges to $f$ in the \wot\ topology.
\end{proof}

The importance of the special assumption (\ref{eq:special_assumption}) 
is in the following result.

\begin{proposition}\label{prop:special}
Let $V$ be an ideal such that  $[I_V H^2_d] = [J_V H^2_d]$.
Let $\cA_V$ be the norm closure of the polynomials in $\cM_V$. 
Then 
\begin{enumerate}
\item For every $f \in \cA_d$, the compression of $M_f$ to $\cF_V$ 
is equal to $M_g$, where $g = f|_V$.
\item $\cA_V = \{f|_V : f \in \cA_d \}$.
\item $\cA_d/I_V$ is completely isometrically isomorphic to $\cA_V$ 
via the restriction map $f \mapsto f|_V$ of $\cA_d$ into $\cA_V$.
\item For every $f \in \cA_d$, $\dist(f,I_V) = \dist(f,J_V)$.
\end{enumerate}
\end{proposition}

\begin{proof}
The first item is just a restatement of Proposition \ref{prop:complete_quotient}. 
By universality of $\cA_d$, $\cA_V$ is equal to the compression of $\cA_d$ to $\cF_V$. 
Therefore, by (a slight modification of) Popescu's results \cite{Popescu06}, 
$\cA_V$ is the universal operator algebra generated by a commuting row 
contraction subject to the relations in $I_V = J_V \cap \cA_d$. 
But so is $\cA_d/ I_V$. So these two algebras can be naturally identified. 
Since compression is restriction, (2) and (3) follow. Item (4) follows from the fact that 
\begin{align*}
 \dist(f,I_V) &= \|f + I_V\|_{\cA/I_V} = 
 \|P_{\cF_V}M_fP_{\cF_V}\| \\&= 
 \|f + J_V\|_{\cM/ J_V} = \dist(f,J_V) . \qedhere
\end{align*}
\end{proof}

\begin{corollary}
Let $V$ be a homogeneous variety, 
or a Blaschke sequence in the disc such that $S(V)$ has measure zero. 
Then $\cA /I_V$ embeds into $\cM / J_V$ isometrically.
\end{corollary}

Define
\[
 \ol{V}^{\cA} = \{\lambda \in \ol{\mb{B}_d} : f(\lambda) = 0 \FORAL f \in I_V  \}.
\]
Clearly $\ol{V}^{\cA}$ contains the closure of $V$ in $\mb{B}_d$.
But it is not clear exactly what else it contains. 
However, it seems most reasonable to restrict our attention to 
the algebras $\cA_V$ such that $V = V(I_V)$, so that the variety 
$V$ is determined by functions in $\cA$. 
In this case, we obtain
\be\label{eq:VA}
 \mb{B}_d \cap \ol{V}^{\cA} = V.
\ee
The proof is the same as that of Proposition \ref{prop:zero_set}. 
It is not clear whether this holds for arbitrary varieties. 
This equation does hold when $V \subseteq \mb{D}$ is a 
Blaschke sequence and $|S(V)|=0$.

\begin{proposition}\label{prop:charA_V}
Let $V$ be a variety satisfying condition $(\ref{eq:special_assumption})$. 
Then the character space $M(\cA_V)$ of $\cA_V$ can be identified with $\ol{V}^{\cA}$.
\end{proposition}

\begin{proof}
Let $\lambda \in \ol{V}^{\cA}$. 
Then the evaluation functional $\rho_\lambda$ given by 
$\rho_\lambda(f) = f(\lambda)$ is a character of $\cA$ with 
kernel equal to $I_{\{\lambda\}} \supseteq I_V$.
Thus $\rho_\lambda$ can be promoted to a character of $\cA_V = \cA/I_V$.

Denote by $Z_1, \ldots, Z_d$ the images of the coordinate functions in $\cA_V$. 
If $\rho$ is a character of $\cA_V$, let 
\[ \lambda = (\lambda_1, \ldots, \lambda_d) = (\rho(Z_1), \ldots, \rho(Z_d)) . \]
Then $\lambda \in \ol{\mb{B}}_d$ because $\rho$ is completely contractive. 
For every $f \in I_V$, $f(Z_1, \ldots, Z_d) = 0$. 
Thus 
\[ \rho(f(Z_1, \ldots, Z_d)) = f(\lambda_1, \ldots, \lambda_d) = 0 .\]
So $\lambda$ lies in the set of all points in $\ol{\mb{B}}_d$ that annihilate $I_V$, 
which is $\ol{V}^{\cA}$.

This identification is easily seen to be a homeomorphism.
\end{proof}

\begin{proposition}\label{prop:alghomo_holo}
Let $V \subseteq \mb{B}_d$ and $W \subseteq \mb{B}_{d'}$ be varieties 
which satisfy condition $(\ref{eq:special_assumption})$. 
Let $\phi : \cA_V \rightarrow \cA_W$ be a unital algebra homomorphism. 
Then there exists a holomorphic map $F: \mb{B}_{d'} \rightarrow \mb{C}^d$ 
that extends continuously to $\ol{\mb{B}}_{d'}$ such that 
\[
 F|_{\ol{W}^{\cA}} = \phi^*.
\]
The components of $F$ are in $\cA_{d'}$, and norm of $F$ as a 
row of multipliers is less than or equal to the cb-norm of $\phi$. 
Moreover, $\phi$ is given by composition with $F$, that is
\[
\phi(f) = f \circ F \qfor f \in \cA_V .
\]
\end{proposition}

\begin{proof}
Every character in $M(\cA_W)$ is an evaluation functional at 
some point $\lambda \in \ol{W}^{\cA}$. 
Identifying $\ol{W}^{\cA}$ and $M(\cA_W)$, we find, as in 
Lemma \ref{L:cc_homo_holo}, that the mapping $\phi^*$ is given by 
\[
 \phi^*(\lambda) = (\phi(Z_1) (\lambda), \ldots, \phi(Z_d) (\lambda))
 \qforal  \lambda \in \ol{W}^{\cA} .
\]
Proposition \ref{prop:special} implies that $\phi(Z_1), \ldots, \phi(Z_d)$ 
are restrictions to $W$ of functions $f_1, \ldots, f_d$ in $\cA_{d'}$.
(This is only true under our special assumption (\ref{eq:special_assumption}).
Otherwise we only get $f_1, \ldots, f_d$ in $\cM_{d'}$). Defining 
\[
 F(z) = (f_1(z), \ldots, f_d(z)), 
\]
we obtain the required map $F$.
Finally, for every $\lambda \in \ol{W}^{\cA}$, 
\[
 \phi(f) (\lambda) = \rho_\lambda (\phi(f)) =
 \phi^*( \rho_\lambda) (f) = \rho_{F(\lambda)} (f)  = f(F (\lambda)) .
\]
Therefore $\phi(f) = f \circ F$.
\end{proof}

This immediately yields:

\begin{corollary}\label{cor:algiso_biholo}
Let $V \subseteq \mb{B}_d$ and $W \subseteq \mb{B}_{d'}$ be varieties 
satisfying condition $(\ref{eq:special_assumption})$. 
If $\cA_V$ and $\cA_W$ are isomorphic, then there are two holomorphic 
maps $F: \mb{B}_{d'} \rightarrow \mb{C}^d$ and 
$G: \mb{B}_d \rightarrow \mb{C}^{d'}$ which extend continuously to the closed balls, 
such that $F(\ol{W}^{\cA}) = \ol{V}^{\cA}$, $G(\ol{V}^{\cA}) = \ol{W}^{\cA}$, 
and $F|_{\ol{W}^{\cA}}$ and $G|_{\ol{V}^{\cA}}$ are inverses of each other.
If $V$ and $W$ satisfy the condition $(\ref{eq:VA})$, 
then $F(W) = V$ and $G(V) = W$.
\end{corollary}

{}From these results and the techniques of 
Lemma \ref{L:comp_iso_iso_biholo}, we also get if $\cA_V$ 
and $\cA_W$ are completely isometrically isomorphic, 
then there exists an automorphism $F \in \Aut(\mb{B}_d)$ 
such that $F(V) = W$. 
On the other hand, the completely isometric isomorphisms of 
Proposition \ref{prop:auto_implemt_iso} are easily seen to 
respect the norm closures of the polynomials in $\cM_V$ and $\cM_W$. 
Together with the above corollary we obtain the following analogue to 
Theorem \ref{theorem:complete_iso_iff_auto}.

\begin{theorem}\label{theorem:complete_iso_iff_auto2}
Let $V$ and $W$ be varieties in $\mb{B}_d$ satisfying $(\ref{eq:special_assumption})$. 
Then $\cA_V$ is completely isometrically isomorphic to $\cA_W$ 
if and only if there exists an automorphism $F$ of $\mb{B}_d$ such that $F(W) = V$. 

Every completely isometric isomorphism $\phi: \cA_V \rightarrow \cA_W$ 
arises as composition $\phi(f) = f \circ F$ where $F$ is such an automorphism. 
In this case $\phi$ is unitarily implemented by a unitary sending the kernel 
function $k_w \in \cF_W$ to a scalar multiple of the kernel function $k_{F(w)} \in \cF_V$.
\end{theorem}

\section{Quotients of $\AD$}

Let $V = \{z_n : n \ge 1 \}$ be a Blaschke sequence in the disk.
Write $B_V$ for the Blaschke product with simple zeros at the points in $V$.
Observe that $J_V = B_V H^\infty$ and $I_V = J_V \cap \AD$.
By Lemma \ref{lem:indisc} and Proposition \ref{prop:special}, 
if the measure $|S(V)|$ of $S(V) = \ol{V}\cap\bT$ is zero, 
then $\A_V = \AD|_{V} \cong \AD/I_V$. 

The interpolating sequences play a special role.

\pagebreak[3]
\begin{theorem} \label{T:AV}
Let $S(V) = \ol{V}\cap\bT$.  
\begin{enumerate}
\item If $|S(V)|>0$ then $I_V=\{0\}$.
\item If $V$ is interpolating and $|S(V)|=0$, then $\A_V$
is isomorphic to $\rC(\ol{V})$ by the restriction map.
\item If $\A_V$ is isomorphic to $\rC(\ol{V})$ via the restriction map, 
then $V$ is an interpolating sequence.
\end{enumerate}
\end{theorem}

\begin{proof}
(1) If $|S(V)|>0$, then any $f\in I_V$ must vanish on $S(V)$, and hence is 0.

(2) The map taking $f\in\AD$ to $f|_{\ol{V}}$ is clearly a contractive homomorphism
of $\AD$ into $\rC(\ol{V})$ with kernel $I_V$.  
So it factors through $\A_V$, and induces an injection of $\A_V$ into $\rC(\ol{V})$.  
It suffices to show that this map is surjective, for the result then 
follows from the open mapping theorem. 

Fix $h \in \rC(\ol{V})$.  By Rudin's Theorem (see \cite[p.81]{Hoffman}), there is a function 
$f\in\AD$ such that $f|_{S(V)} = h|_{S(V)}$.  
By replacing $h$ with $h-f$, we may suppose that $h|_{S(V)} = 0$.  
Hence $h|_V$ is a function that $\lim_{n\to\infty} h(z_n) = 0$.
Now it suffices to show that if $h(z_n) = 0$ for all $n>N$, then there is
a function $f\in\AD$ with $f|_V = h|_V$ and $\|f\| \le C \|h\|_\infty$ for
a constant $C$ which is independent of $N$.
Surjectivity will follow from a routine approximation argument.
Let $c$ be the interpolation constant for $V$.

Fix $N$. Take $h\in \rC(\ol{V})$ with $h(z_n) = 0$ for all $n>N$ and $\|h\|_\infty \le 1$.
By a theorem of Fatou \cite[p.81]{Hoffman}, there is an analytic function $g$ on $\bD$
such that $\re g \ge0$ and $e^{-g} \in \AD$ vanishes precisely on $S(V)$.
There is an integer $m>0$ so that $|e^{-g/m}(z_n)| > .5$ for $1 \le n \le N$.
Set $V_N = \{ z_n : n > N \}$.
Since $V$ is interpolating, 
\[ \min \{ |B_{V_N}(z_n)| : 1 \le n \le N \} \ge 1/c .\]
We will look for a function $f$ of the form $f = B_{V_N} e^{-g/m} f_0$.
By the arguments for Rudin's theorem, this will lie in $\AD$.
Clearly it vanishes on $V_N \cup S(V)$, and we require
\[ h(z_n) = f(z_n) = B_{V_N}(z_n) e^{-g/m}(z_n) f_0(z_n) .\]
So we need to find $f_0\in\AD$ with $\|f_0\| \le C$ and 
\[ f_0(z_n) = a_n := h(z_n)e^{g(z_n)/m} / B_{V_N}(z_n) \qfor 1 \le n \le N . \]
The estimates made show that $|a_n| \le 2c$.
The interpolation constant for $\{z_n : 1 \le n \le N \}$ is at most $c$,
and since this is a finite set, we can interpolate using functions in $\AD$
which are arbitrarily close to the optimal norm.  Thus we can find an $f_0$
with $\|f_0\| \le 3c^2$.  Hence we obtain $f$ with the same norm bound.

(3) If $\A_V$ is isomorphic to $\rC(\ol{V})$ via the restriction map, then by the open
mapping theorem, there is a constant $c$ so that for any $h\in\rC(\ol{V})$,
there is an $f\in\AD$ with $f|_V = h|_V$ and $\|f\| \le c \|h\|$.
In particular, for any bounded sequence $(a_n)$ and $N \ge 1$, 
there is an $f_N \in \AD$ such that $\|f_N\|\le c \|(a_n)\|_\infty$ and 
\[ f_N(z_n) = \begin{cases} a_n &\IF 1 \le n \le N\\ 0 &\IF n > N \end{cases} .\]
Take a weak-$*$ cluster point $f$ of this sequence in $\Hinf$.
Then $\|f\| \le c$ and $f$ interpolates the sequence $(a_n)$ on $V$.
So $V$ is interpolating. 
\end{proof}

We can now strengthen Example \ref{ex:interp_non_interp_hol}, 
showing that there are discrete varieties giving rise to 
non-isomorphic algebras which are biholomorphic with a 
biholomorphism that extends continuously to the boundary.

\begin{eg}
We will show that there is a Blaschke sequence $V$ which is not interpolating
and an interpolating sequence $W$ and functions $f$ and $g$ in $\AD$ so that
$f|_V$ is a bijection of $V$ onto $W$ and $g|_W$ is its inverse.
Take
\[ V = \{ z_n := 1-n^{-2} : n \ge 1 \} \AND W = \{ w_n := 1- n^{-2}e^{-n^2} : n \ge 1 \} .\]
Then $W$ is an interpolating sequence, and $V$ is not.  Let 
\[ f(z) = 1 + (z-1) e^{1/(z-1)} .\]
Then since $1/(z-1)$ takes $\bD$ conformally onto $\{z : \re z < -1/2 \}$,
it is easy to see that $e^{1/(z-1)}$ is bounded and continuous on $\ol{\bD}\setminus\{1\}$.
Hence $f$ is continuous, so lies in $\AD$.
Clearly, \mbox{$f(z_n) = 1- n^{-2}e^{-n^2} = w_n$} for $n\ge1$.
The inverse of $f|_V$ is the map $h(w_n) = z_n$.  
Since 
\[ \lim_{n\to\infty} h(w_n) = \lim_{n\to\infty} z_n = 1 ,\]
this extends to be a continuous function on $\ol{W} = W \cup\{1\}$.
By Theorem~\ref{T:AV}, there is a function $g\in\AD$ such that $g|_W = h$.
\end{eg}

\begin{rem}
Let $V = \{v_n\}$ and $W =  \{w_n\}$ be two interpolating sequences in $\bD$
with $\lim v_n = \lim w_n = 1$. 
Then the algebras $\A_V$ and $\A_W$ are both isomorphic to $c$, the space
of convergent sequences. As in our counterexamples using Blaschke products,
we can find biholomorphisms carrying one sequence onto the other. 
However there is no reason for the rates at which they approach the 
boundary to be comparable.
\end{rem}

We now give a strengthening of Theorem \ref{T:AV}.
\begin{theorem}
Let $V = \{v_n\}_{n=1}^\infty$ be a Blaschke sequence in $\mb{D}$, 
such that $|S(V)| = 0$. 
Then $\cA_V$ is isomorphic to $\rC(\ol{V})$ if and only if $V$ is interpolating.
\end{theorem}

\begin{proof}
Theorem \ref{T:AV} says that $\cA_V$ is isomorphic to $\rC(\ol{V})$ 
via the restriction map if and only if $V$ is interpolating. 
All that remains to prove is that if $V$ is not an interpolating 
sequence, then it cannot be isomorphic via any other isomorphism. 

Suppose that $V$ is a non-interpolating sequence and define 
$w_n = (1-e^{-n})v_n/|v_n|$. 
Then $W = \{w_n\}$ is an interpolating sequence with $S(W) = S(V)$, 
and $\ol{V}$ is homeomorphic to $\ol{W}$ via the map that
continuously extends $v_n \mapsto  (1-e^{-n})v_n/|v_n|$. 
Therefore, $\cA_W$ is isomorphic to $\rC(\ol{V})$ via the restriction map. 
Now assume that $\cA_V$ is isomorphic to $\rC(\ol{V})$ by any isomorphism. 
Then it is isomorphic to $\cA_W$. But by Corollary \ref{cor:algiso_biholo}, 
this isomorphism is given by 
composition with a holomorphic map. 
Therefore $\cA_V$ is isomorphic to $\rC(\ol{V})$ via the restriction map---a contradiction.
\end{proof}

\begin{remark}
In \cite{DRS}, now improved by \cite{Hartz}, 
we saw that in the case of homogeneous varieties 
$V$ and $W$, the algebras $\cA_V$ and $\cA_W$ are isomorphic
if and only if the algebras $\cM_V$ and $\cM_W$ are isomorphic. 
The above discussion shows that this is not true in general.  
If $V$ and $W$ are two interpolating sequences in $\mb{D}$,
then $\cM_V$ and $\cM_W$ are both isomorphic to $\ell^\infty$, 
whereas the isomorphism classes of $\cA_V$ and $\cA_W$ 
depend on the structure of the limit sets.
\end{remark}

\bibliographystyle{amsplain}

\end{document}